\documentclass[12pt]{amsart}
\usepackage{enumerate}
\usepackage{a4wide}
\usepackage{graphicx}
\usepackage{etoolbox}
\patchcmd{\subsubsection}{\itshape}{\bfseries\itshape}{}{}

\usepackage{amsmath,amsthm,amssymb,latexsym,amsfonts,array}
\usepackage{xcolor}
\newtheorem{theorem}{Theorem}[section]
\newtheorem{lemma}[theorem]{Lemma}
\newtheorem{proposition}[theorem]{Proposition}

\newtheorem{corollary}[theorem]{Corollary}

\newtheorem*{theoremA}{Theorem A}
\newtheorem*{theoremB}{Theorem B}

\theoremstyle{remark}
\newtheorem{remark}[theorem]{Remark}
\theoremstyle{remark}

\newcolumntype{M}[1]{>{\centering\arraybackslash}m{#1}}
\newcolumntype{N}{@{}m{0pt}@{}}

\newcommand{\C}{\ensuremath{\mathbb{C}}}
\newcommand{\R}{\ensuremath{\mathbb{R}}}
\renewcommand{\H}{\ensuremath{\mathbb{H}}}
\renewcommand{\O}{\ensuremath{\mathbb{O}}}

\newcommand{\g}[1]{\ensuremath{\mathfrak{#1}}}

\newcommand{\s}[1]{\ensuremath{\mathsf{#1}}}

\DeclareMathOperator{\tr}{tr}
\DeclareMathOperator{\proj}{proj}

\DeclareMathOperator{\Isom}{Isom}
\DeclareMathOperator{\Id}{Id}

\DeclareMathOperator{\Ad}{Ad}
\DeclareMathOperator{\Aut}{Aut}
\DeclareMathOperator{\ad}{ad}

\DeclareMathOperator{\Ker}{Ker}
\DeclareMathOperator{\spann}{span}
\DeclareMathOperator{\rank}{rank}

\begin{document}
	\title[Totally geodesic submanifolds in products of rank one symmetric spaces]{Totally geodesic submanifolds in\\ products of rank one symmetric spaces}

\author[A. Rodr\'{\i}guez-V\'{a}zquez]{Alberto Rodr\'{\i}guez-V\'{a}zquez}
\address{KU Leuven, Department of Mathematics, Celestijnenlaan 200B, Leuven, Belgium}
\email{alberto.rodriguezvazquez@kuleuven.be}
\thanks{The author has been supported by the projects PID2022-138988NB-I00 funded by MICIU/AEI/10.13039/501100011033 and by ERDF (European Union); and ED431F 2020/04 (Xunta de Galicia, Spain), the Methusalem grant METH/15/026 of the Flemish Government (Belgium), and the FWO Postdoctoral grant with project number 1262324N}



\begin{abstract}
In this article, we classify totally geodesic submanifolds in arbitrary products of rank one symmetric spaces. Furthermore, we give infinitely many examples of irreducible totally geodesic submanifolds in Hermitian symmetric spaces with non-trivial constant K\"ahler angle, i.e.\  they are neither complex nor totally real.
\end{abstract}


\maketitle
\section{Introduction}
\label{sect:intro}
The investigation of totally geodesic submanifolds in symmetric spaces has been a relevant and significant topic of research in submanifold geometry in the last decades. The classification of totally geodesic submanifolds in symmetric spaces was started in 1963 by Wolf \cite{wolfrankone}, who classified these objects in symmetric spaces of rank one.

Throughout this article, we will extend Wolf's result to products of rank one symmetric spaces.
In the non-compact setting these are products of hyperbolic spaces $\mathbb{F}\s H^n$, where $\mathbb{F}\in\{\R,\C,\H,\O\}$.

In this article, we  introduce some slight modification of Young tableaux that we will call adapted Young tableaux  (see \S\ref{sect:reduciblerank2} for the definition), which will be useful to classify totally geodesic submanifolds in arbitrary products of symmetric spaces of rank one and to determine their isometry type via Corollary~\ref{cor:sympolynomial}. Young tableaux have been used to classify irreducible representations of the symmetric group and they provide an effective way to gain understanding of a given irreducible representation. In a similar vein, we will exhibit that Young tableaux can also be useful to understand totally geodesic submanifolds in products of symmetric spaces of rank one. To illustrate this,  let us consider 
 \[M=\C \s H^2\times \C \s  H^2 \times \H \s H^3 \times \O  \s H^2 \times \R \s H^4.\]
  Here, the metric of $\mathbb{F}\s H^n$ is such that its sectional curvature is equal to $-1$, when $\mathbb{F}=\R$, and pinched between $-1$ and $-1/4$, in the other cases. Let $c$ be a positive number. If a symmetric space of rank one has constant sectional curvature equal to $-c$, we  denote it by $\R \s H^n(c)$; and by $\mathbb{F}\s H^n(c)$, $\mathbb{F}\in\{\C,\H,\O  \}$, when its sectional curvature lies in $[-c,-c/4]$.

\begin{figure}
	\label{fig:firstex}
 	\renewcommand{\arraystretch}{1.3}
 	\begin{tabular}{|l|ll}
 		\cline{1-2}
 		$\mathbb{R}\s H^2\subset\mathbb{C} \s H^2$ & \multicolumn{1}{l|}{$\mathbb{R}\s H^2(\frac{1}{2})\subset\mathbb{C} \s H^2$} & $\mathbb{R}\s H^2(\frac{1}{5})$        \\ \cline{1-2}
 		$\mathbb{H}\s H^2\subset\mathbb{H}\s H^3$               & \multicolumn{1}{l|}{$\mathbb{H}\s H^2\subset\O \s H^2$}         & $\mathbb{H}\s H^2(\frac{1}{2})$ \\ \cline{1-2}
 		$\mathbb{R} \s H^4\subset\mathbb{R}\s H^4$              & $\mathbb{R} \s H^4$                                          &                                     \\ \cline{1-1}
 	\end{tabular}

 \caption{An example of a Young tableau $T$ adapted to the product of rank one symmetric spaces  $M=\C \s H^2\times\C \s H^2\times \H \s H^3\times \O \s H^2\times \R \s H^4$.}
\end{figure}

 In Figure~\ref{fig:firstex} we show an example of a Young tableau $T$ adapted to $M$. In each block of $T$ there is an embedding of a totally geodesic submanifold of a factor of $M$. We can attach to $T$ a semisimple totally geodesic submanifold $\Sigma_T$ of $M$. In this example, the totally geodesic submanifold $\Sigma_T$ is isometric to $\R \s H^2(\frac{1}{5})\times \H \s H^2(\frac{1}{2})\times \R\s H^4$. Each row of $T$ describes a factor of $\Sigma_T$, thus the number of rows of $T$ is equal to the number of factors of $\Sigma_T$. The number of blocks in the $i$-th row indicates the number of factors of $M$ in which the $i$-th factor of $\Sigma_T$ lies ``diagonally'' (see Subsection~\ref{sect:preliminaries2}). Furthermore, the isometry type of each factor of $\Sigma_T$ is computed using the formula involving elementary symmetric polynomials appearing in Corollary~\ref{cor:sympolynomial}.

In this article, we prove a result (Proposition \ref{prop:youngtableaux}) which gives a  correspondence between these  adapted Young tableaux and semisimple totally geodesic submanifolds in products of rank one symmetric spaces. This will lead us to the following theorem, which via duality gives a classification of totally geodesic submanifolds in arbitrary products of symmetric spaces of rank one.
\begin{theoremA}
Let $M= M_1\times\cdots\times M_r$, where $M_i$ is a symmetric space of non-compact type and rank one for each $i\in\{1,\ldots, r\}$.
Then, a submanifold $\Sigma$ of $M$ is totally geodesic if and only if $\Sigma=\Sigma_0\times \Sigma_T$,  where $\Sigma_T$ is a semisimple totally geodesic submanifold corresponding to a Young tableau $T$ adapted to $M_{\sigma(1)}\times \cdots\times M_{\sigma(k)}$, $\Sigma_0$ is a flat totally geodesic submanifold of $M_{\sigma({k+1})}\times \cdots \times M_{\sigma(r)}$,  $\sigma$ is any permutation of $\{1,\ldots,r\}$, and $k\in\{1,\ldots,r\}$. 
\end{theoremA}

In the following, we give a brief overview of some known results about totally geodesic submanifolds of symmetric spaces. In 1977, Chen and Nagano \cite{chen1,chen2} gave a classification of totally geodesic submanifolds in irreducible symmetric spaces of rank two. In this classification there were some examples missing that were found by Klein in a series of papers \cite{kleindga,kleintams,kleinosaka}. 
Moreover,  Berndt, Olmos and Rodr\'iguez \cite{BO1,BO2,BO3,BO4,BO5} computed the index of irreducible symmetric spaces. This is the minimal codimension of a proper totally geodesic submanifold of a symmetric space. In particular, they proved that in an irreducible symmetric space $M\neq\mathsf{G}^2_2/\mathsf{SO}_4,\mathsf{G}_2/\mathsf{SO}_4$ the index is realized by a reflective totally geodesic submanifold. In view of the known results, Theorem A provides the first classification of (not necessarily maximal) totally geodesic submanifolds in some symmetric space of rank higher than two. However, Remark~\ref{remark:bigrk} points out that we cannot expect a result similar to Theorem~A  for products of irreducible spaces of rank greater than one.

A special class of symmetric spaces where totally geodesic submanifolds can be studied is that of  Hermitian symmetric spaces. On the one hand, we can use the notion of K\"ahler angle \cite{damekricci,mathz} to measure how a submanifold fails to be complex in a Hermitian symmetric space; see Section~\ref{sect:slant} for the definition of K\"ahler angle. For example, a totally geodesic submanifold is complex or totally real if and only if it has constant K\"ahler angle equal to $0$ or to $\pi/2$, respectively. On the other hand, complex totally geodesic submanifolds in Hermitian symmetric spaces were classified by Ihara in \cite{ihara}, and real forms, which constitute a particular type of totally real and totally geodesic submanifolds, were studied and  classified by Jaffe \cite{jaffee1,jaffee2}  and Leung \cite{leung,leung2}.

While in complex projective spaces (the Hermitian symmetric spaces of compact type and rank one) a totally geodesic submanifold is either complex or totally real, in the rank two case the situation is more involved. Klein found two examples of irreducible totally geodesic submanifolds which are neither complex nor totally real, one example in the complex quadric and another one in the complex 2-plane Grassmannian. These  examples have   non-trivial constant K\"ahler angle, i.e.\ they have constant K\"ahler angle different from $0$ and $\pi/2$, and they have been the only known totally geodesic submanifolds with non-trivial constant K\"ahler angle up to the present. In particular, it can be computed that these two examples have constant K\"ahler angle equal to $\arccos(1/5)$. 

In this article, we will give a method to construct infinitely many examples of irreducible totally geodesic submanifolds which have non-trivial constant K\"ahler angle in irreducible Hermitian symmetric spaces of higher rank. This method will rely on the construction of certain totally geodesic submanifolds contained in a product of Hermitian symmetric spaces. Clearly, $\C^n$, $n\ge 2$, is a Hermitian symmetric space where every $\varphi\in[0,\pi/2]$ can be realized as the constant K\"ahler angle of some totally geodesic submanifold. Therefore, we will exclude from our study flat Hermitian symmetric spaces. Let $\mathcal{J}_r$ be the set of K\"ahler angles of totally geodesic submanifolds in irreducible non-flat Hermitian symmetric spaces of rank $r$, and $\mathcal{J}=\bigcup_{r\ge 1} \mathcal{J}_r$ the set of K\"ahler angles of  totally geodesic submanifolds in non-flat irreducible Hermitian symmetric spaces. As a consequence of the classification results in rank one and two, we have 
\[ \mathcal{J}_1=\{0,\pi/2\}, \quad \mathcal{J}_2=\{0,\arccos(1/5),\pi/2\}.   \]
Consequently, an interesting problem is to compute $\mathcal{J}$ or $\mathcal{J}_r$.
In this article, we will prove the following result.
\begin{theoremB}
	\label{th:cqa}
	Let  $\mathcal{J}$ be the set of  K\"ahler angles of totally geodesic submanifolds in non-flat irreducible  Hermitian symmetric spaces. Then  $\mathcal{J}$ contains a dense subset of $[0,\pi/2]$.
\end{theoremB}

This article is structured as follows. In Section \ref{sect:preliminaries}, we start by recalling some known results concerning totally geodesic submanifolds in reducible symmetric spaces. In particular, we introduce the notions of $k$-diagonal and diagonal totally geodesic submanifolds. Then, we give a structure result for totally geodesic submanifolds of maximal rank in reducible symmetric spaces (see~Proposition~\ref{prop:maximalrank}).

In Section~\ref{sect:reduciblerank2}, we prove Theorem~A, which gives a classification of totally geodesic submanifolds in a product of an arbitrary number of rank one symmetric spaces $M=M_1\times\cdots\times M_r$. The proof is organized in the following manner.  Firstly, by Lemma~\ref{lemma:kblock}, every irreducible totally geodesic submanifold $\Sigma$ of $M$ must be $k$-diagonal for some $k\in\{1,\ldots,r\}$. Moreover, by Proposition~\ref{proposition:compactnoncompact}, a diagonal totally geodesic submanifold of a product of symmetric spaces of opposite type must be flat. Thus, we will assume that $M$ is a product of hyperbolic spaces over $\R$, $\C$, $\H$, and $\O$. Secondly, in Theorem~\ref{th:rank2reducible}, we classify totally geodesic submanifolds of a product of two hyperbolic spaces. As a consequence, we compute all the possible isometry types of $k$-diagonal irreducible totally geodesic submanifolds in an arbitrary product of hyperbolic spaces (see~Corollary~\ref{cor:sympolynomial}). Thirdly, in Lemma~\ref{lemma:proj}, we prove that if $\Sigma$ is a reducible semisimple totally geodesic submanifold of $M=M_1\times\cdots\times M_r$, then different factors of $\Sigma$ must be contained in different factors of $M$. This property allows to parametrize semisimple totally geodesic submanifolds using Young Tableaux adapted to~$M$, as proved in Proposition~\ref{prop:youngtableaux}. Then, Theorem~A follows from this proposition.

Finally, in Section~\ref{sect:slant}, we construct totally geodesic submanifolds with non-trivial constant K\"ahler angle in complex Grassmannians, and prove Theorem B. In order to prove this, firstly, we consider natural embeddings of products of complex projective spaces in complex Grassmannians. Then, we construct certain diagonal totally geodesic submanifolds in these products by combining holomorphic and antiholomorphic isometries.

\textbf{Acknowledgments}. This work was started during a stay at King's College London in Spring 2020. I am deeply indebted to my host Prof.~Jürgen Berndt for enlightening discussions. I would also like to thank Prof.~Sebastian Klein for helpful comments, and my advisors Prof.~José Carlos Díaz-Ramos and Prof.~Miguel Domínguez-Vázquez for their careful reading of the manuscript and their constant support.

\section{Totally geodesic submanifolds in reducible symmetric spaces}
\label{sect:preliminaries}
Let us start this section by recalling some basic facts  about Riemannian symmetric spaces and their totally geodesic submanifolds. 
\subsection{Totally geodesic submanifolds in  symmetric spaces}
\label{sect:preliminaries1}
Let $M=\s G/\s K$ be a connected Riemannian symmetric space, where $\s G$ is, up to some finite covering, equal to $\mathrm{Isom}^0(M)$, the connected component of the identity of the isometry group of $M$, and $\s K$ is the isotropy at some fixed point $o\in M$.  Let $\theta$ be a Cartan involution of $\g{g}$, the Lie algebra of $\s G$. Then, $\theta$ is a Lie algebra automorphism of $\g{g}$ which has $1,-1$ as  eigenvalues and their corresponding eigenspaces are $\g{k}$, the Lie algebra of $\s K$, and a subspace $\g{p}$ of $\g{g}$, respectively. Hence, we can decompose $\g{g}$ as $\g{g}=\g{k}\oplus \g{p}$.  Furthermore, we may identify $\g{p}$ with $T_o M$,  and we have the following bracket relations
\[[\g{k},\g{k}]\subset \g{k}, \quad [\g{k},\g{p}]\subset \g{p}, \quad [\g{p},\g{p}]\subset \g{k}.   \]
All maximal abelian subspaces of $\g{p}$ have the same dimension, which is called the rank of the symmetric space $M$. Let $\mathcal{B}$ be the Killing form of $\g{g}$. It turns out that $\g{k}$ and $\g{p}$ are orthogonal with respect to $\mathcal{B}$ and that every automorphism of $\g{g}$ is a linear isometry for $\mathcal{B}$.
If we consider the linearization at $o\in M$ of the isotropy action of $\s K$ on $M$, we get the \textit{isotropy representation} of the symmetric space $M$, which is defined as $k\in \s K\mapsto k_{*o}\in \mathsf{GL}(T_o M)$. This is equivalent to the adjoint representation of $\s G$ restricted to $\s K$ on $\g{p}$ via the identification of $\g{p}$ and $T_o M$. A symmetric space is \textit{irreducible} if the universal cover of $M$, which is again a symmetric space, is not isometric to  a non-trivial product of symmetric spaces. Otherwise, $M$ is said to be \textit{reducible}. 
A symmetric space is said to be of \textit{compact type}, \textit{non-compact type} or \textit{Euclidean type} if $\mathcal{B}_{|\g{p}\times \g{p}}$, the restriction of the Killing form $\mathcal{B}$ to $\g{p}$, is negative definite, positive definite or identically zero, respectively. Let
 $\widetilde{M}$ be the universal cover of $M$. Then, by the De-Rham Theorem,  it splits as $\widetilde{M}=M_0 \times M_{+} \times M_{-}$, where $M_0$ is isometric to $\mathbb{R}^n$ and is called the flat factor, and $M_{+}$ and $M_{-}$ are simply connected symmetric spaces of compact and non-compact type, respectively.  If $M= \s G/\s K$ is a symmetric space of non-compact type, we consider the positive definite inner product on $\g{g}$ given by $\langle X, Y\rangle:=-\mathcal{B}(X,\theta Y)$, for every $X,Y\in \g{g}$.

Let $\Sigma$ be a totally geodesic submanifold of $M=\s G/\s K$. We will assume throughout this paper that totally geodesic submanifolds are connected and complete since every connected totally geodesic submanifold is contained in a complete one. By the homogeneity of $M$, we can assume without loss of generality that $o\in \Sigma\subset M$. Cartan proved in \cite{Cartan} that a totally geodesic submanifold $\Sigma$ of $M$ with $o\in \Sigma$ and $T_o \Sigma=V\subset T_o M$ exists if and only if $V\subset T_o M$ satisfies $R_o(V,V)V\subset V$, where $R$ is the Riemannian curvature tensor of $M$. Using the identification of $\g{p}$ and $T_o M$, we can write the curvature tensor of $M$ at $o$ as
\[ R_o(X,Y)Z=-[[X,Y],Z], \]
for $X,Y,Z\in T_o M$. Thus, a subspace $V\subset \g{p}$ is curvature invariant if and only if $[[X,Y],Z]\in V$ for every $X,Y,Z\in V$. A subspace $V$ of $\g{p}$ with this property is a \textit{Lie triple system} in $\g{p}$. Hence, every totally geodesic submanifold of a symmetric space is again symmetric and there is a one-to-one correspondence between Lie triple systems $V$ in $\g{p}$ and totally geodesic submanifolds $\Sigma$ in $M$ containing $o\in M$.

An important notion which establishes a relation between symmetric spaces of compact type and non-compact type is duality. If we restrict our attention to simply connected symmetric spaces, there is a one-to-one correspondence between symmetric spaces of non-compact type and symmetric spaces of compact type. 
Moreover, totally geodesic submanifolds are preserved under duality. Hence, when studying totally geodesic submanifolds,   it will not be restrictive to assume that our ambient symmetric space, if it does not have  local flat factors, is either of compact type or of non-compact type.

\subsection{Diagonal totally geodesic submanifolds}
\label{sect:preliminaries2}
In what follows, we will introduce the notions of $k$-diagonal linear subspace and  of $k$-diagonal totally geodesic submanifold. 

Let $V$ be a vector space equipped with a positive definite scalar product and let us choose subspaces $V_i\subset V$ such that $V=\bigoplus_{i=1}^r V_i$ is an orthogonal decomposition of $V$, and denote by $\proj_i\colon V\rightarrow V_i$ the orthogonal projection onto $V_i$. We say that a subspace $W\subset V$ is \textit{$k$-diagonal} with respect to the above decomposition if there is a collection of indexes $\{i_1,\ldots,i_k\}\subset\{1,\ldots,r\}$ such that every non-zero element $w$ of $W$ satisfies $\proj_l w\neq 0$ when $l\in\{i_1,\ldots,i_k\}$; and reciprocally, if  $l\not\in\{i_1,\ldots, i_k \}$, we have $\proj_l w=0$. For instance, let $V:=\R^3$ and consider $V_i:=\spann\{e_i\}$, where $\{e_i\}_{i=1}^3$ is the canonical basis of $\R^3$. Then, we have that $W=\spann\{e_1+e_2\}$ is a $2$-diagonal subspace and $W'=\spann\{e_1 + e_2, e_1+e_3\}$ is not $k$-diagonal for any $k\in\{1,2,3\}$.

Let $M$ be a complete connected Riemannian manifold. Then, by the De-Rham Theorem, the universal cover of $\widetilde{M}$ of $M$ splits as a Riemannian product $\widetilde{M}={M}_0\times{M}_1\times\cdots\times {M}_r$, where ${M}_0$ is a Euclidean space and ${M}_i$ is a connected, complete, simply connected and non-flat Riemannian manifold for each $i\in\{1,\ldots, r\}$. Moreover, this decomposition is unique up to order. Let $p~=~(p_0,\ldots,p_r)\in \widetilde{M}$, $  \Sigma\subset M$  a submanifold, and $\pi\colon \widetilde{M} \rightarrow M$  the universal covering map such that $q~=~\pi(p)~\in~\Sigma$. Then, $T_p \widetilde{M}= \bigoplus_{i=0}^r T_{p_i}{M}_i$, and since $\pi_{*p}$ is a linear isometry,  we have an orthogonal decomposition of $T_q M$ given by 
\begin{equation}
\label{eq:decomposition1}
T_q M=\bigoplus_{i=0}^r \pi_{*p} T_{p_i} {M}_i.
\end{equation}
Then, we say that $\Sigma$ is \textit{$k$-diagonal} at $q\in\Sigma$ if $T_q \Sigma\subset T_q M$ is $k$-diagonal with respect to the  decomposition in Equation (\ref{eq:decomposition1}), and we say that $\Sigma$ is \textit{diagonal} at $q$ if  $\Sigma$ is $k$-diagonal at $q$ for some $k>1$. It is easy to check that this definition does not depend on $p\in\pi^{-1}(q)$.

If our ambient space is homogeneous and $\Sigma$ is extrinsically homogeneous, we  say that $\Sigma\subset M$ is \textit{$k$-diagonal} if there is some point $q\in \Sigma$ such that $T_q \Sigma\subset T_q M$ is $k$-diagonal with respect to the decomposition in Equation~(\ref{eq:decomposition1}). 

In the present article, our interest is on Riemannian symmetric spaces, so let $M=~\s G/\s K$ be a simply connected Riemannian symmetric space. Hence, by De-Rham Theorem, $M=~M_0\times\cdots\times M_r$, where $M_0$ is isometric to some Euclidean space and $M_i$ is a simply connected, irreducible symmetric space for each $i\in\{1,\ldots,r\}$. Moreover, we can identify $T_p M$ with
\begin{equation}
\label{eq:decompositionp}
\g{p}:=\bigoplus_{i=0}^r \g{p}_i,
\end{equation}
where $\g{p}_i$ is a Lie triple system in $\g{p}$, which is identified with the tangent space $T_{p_i} M_i$ of the $i$-th factor in the decomposition of $M$. Observe that $[\g{p}_i,\g{p}_j]=0$, for each $i,j\in \{0,\ldots,r\}$, $i\neq j$, and that $\g{p}_0$ is identified with the flat factor in the decomposition of $M$ given by De-Rham Theorem, so $[\g{p}_0,\g{p}_0]=0$. 
Furthermore, we can define 
\[ \g{k}_i:=[\g{p}_i,\g{p}_i],  \qquad \g{g}_i:=\g{k}_i\oplus \g{p}_i, \]
where $i=1,\ldots,r$. Then, $\g{k}_i$ and $\g{g}_i$ are the Lie algebras corresponding to the isotropy and isometry groups of $M_i$, respectively. On $\g{p}$ (and hence, on each $\g{p}_i$) we will consider the inner product $\langle\cdot,\cdot\rangle$ induced from the metric of $M$ (resp.\ from the metric of $M_i$) via the identification $\g{p}\cong T_o M$ (resp.\ $\g{p}_i\cong T_{o_i}M_i$).

Let $\Sigma\subset M$ be a totally geodesic submanifold through the point $o\in \Sigma\subset M$. Then, there is a Lie triple system $\g{p}_{\Sigma} \subset\g{p}$ corresponding to $\Sigma$. Notice that if $\g{a}_{\Sigma}$ is a maximal abelian subspace of $\g{p}_{\Sigma}$, then it is contained in some maximal abelian subspace $\g{a}$ of $\g{p}$. 
By the discussion above, since totally geodesic submanifolds in symmetric spaces are extrinsically homogeneous, a totally geodesic submanifold $\Sigma\subset M$ is $k$-diagonal if and only if $\g{p}_{\Sigma}\subset \g{p}$ is $k$-diagonal with respect to the decomposition in (\ref{eq:decompositionp}).

Notice that $\proj_i \g{p}_{\Sigma}$ is a Lie triple system in $\g{p}_i$, where $\proj_i\colon \g{p}\rightarrow \g{p}_i$ is the orthogonal projection onto $\g{p}_i$. Indeed, given $X,Y,Z\in\g{p}_{\Sigma}$, we have  
\[ [[\proj_i X,\proj_i Y],\proj_i Z]=\proj_i[[X,Y],Z]\in \proj_i \g{p}_{\Sigma},   \]
since the projection $\proj_i\colon\g{g}\rightarrow \g{g}_i$ is a Lie algebra homomorphism and $\g{p}_{\Sigma}$ is a Lie triple system. Hence, $\g{g}^i_{\Sigma}:=\proj_i \g{p}_{\Sigma} \oplus [\proj_i \g{p}_{\Sigma},\proj_i \g{p}_{\Sigma}]$ is a Lie subalgebra of $\g{g}_i$. Let $\s G_{\Sigma}$ and $\s G^i_{\Sigma}$ be the connected subgroups of $\s G$ with Lie algebras $\g{g}_{\Sigma}:=\g{p}_{\Sigma}\oplus[\g{p}_{\Sigma},\g{p}_{\Sigma}]$ and $\g{g}^i_{\Sigma}$, respectively. Therefore, by \cite[Proposition 11.1.2]{BCO}, we have that $\Sigma=\s G_{\Sigma}\cdot o$ and $\proj_i \Sigma=\proj_i(\s G_{\Sigma}\cdot o)=\s G^i_{\Sigma}\cdot o$ are totally geodesic submanifolds in $M$ and $M_i$, respectively,  where by $\proj_i$ we also denote the projection $M\to M_i$. Thus, we obtain the following useful lemma.
\begin{lemma}
	\label{lemma:projtriple}
	Let $M=M_0\times\cdots\times M_r$ be a product of simply connected symmetric spaces and $\Sigma$ a totally geodesic submanifold of $M$. Then, $\proj_i \Sigma$ is a totally geodesic submanifold of $M_i$ for each $i\in\{0,\ldots,r\}$.
\end{lemma}
The following result gives a sufficient condition for a totally geodesic submanifold $\Sigma\subset M$ to be not diagonal, namely $\rank \Sigma=\rank M$.
\begin{proposition}
\label{prop:maximalrank}
Let $M=M_1\times\cdots\times M_r$ be a product of simply connected irreducible symmetric spaces and $\Sigma\subset M$ a totally geodesic submanifold with the same rank as $M$. Then, $\Sigma=\Sigma_1\times\cdots\times\Sigma_r$, where $\Sigma_i\subset M_i$ is a totally geodesic submanifold.
\end{proposition}
\begin{proof}
Let $\g{p}_{\Sigma}$ be the Lie triple system in $\g{p}$ corresponding to $\Sigma$. We clearly have  $\g{p}_{\Sigma}\subset\bigoplus_{i=1}^r \proj_i \g{p}_{\Sigma}$. Let $X_j\in \proj_j \g{p}_{\Sigma}$ for some $j\in\{1,\ldots,r\}$, and $X\in\g{p}_{\Sigma}$ such that $\proj_j X=X_j$. Let $\g{a}_{\Sigma}\subset \g{p}_{\Sigma}$ be a maximal abelian subspace containing $X$. Since $\Sigma$ has the same rank as $M$, $\g{a}_{\Sigma}$ is a maximal abelian subspace of $\g{p}$. But every maximal abelian subspace of $\g{p}$ is the sum of maximal abelian subspaces $\g{a}_i$ of $\g{p}_i$. Thus, $X=\sum_{i=1}^r X_i$, where $X_i\in\g{a}_i\subset \g{a}_{\Sigma}\subset \g{p}_{\Sigma}$ and $\g{a}_i$ is a maximal abelian subspace of $\g{p}_i$. In particular, $X_j$ belongs to $\g{p}_{\Sigma}$. Since $j\in\{1,\ldots,r\}$ was arbitrary, we have  $\g{p}_{\Sigma}=\bigoplus_{i=1}^r \proj_i \g{p}_{\Sigma}$. By Lemma \ref{lemma:projtriple}, we have $\Sigma=\Sigma_1\times\cdots\times \Sigma_r$, where each $\Sigma_i:=\proj_i \Sigma$ is a totally geodesic submanifold of $M_i$, for each $i\in\{1,\ldots,r\}$.
\end{proof}

\section{Totally geodesic submanifolds in\\ products of symmetric spaces of rank one}
\label{sect:reduciblerank2}
In this section, we will give a classification of  totally geodesic submanifolds in products of symmetric spaces of rank one.

Let us recall the classification of totally geodesic submanifolds in symmetric spaces of rank one. Let $M$ be a symmetric space of non-compact type and rank one. Then, $M$ is either a real hyperbolic space $\R \s H^{n}$, $n\geq 2$, a complex hyperbolic space $\C \s H^{n}$, $n\geq 2$, a quaternionic hyperbolic space $\H \s H^{n}$, $n\geq 2$, or the Cayley hyperbolic plane $\O \s H^2$. We use the notation $\mathbb{F} \s H^n$, where $\mathbb{F}\in \{\R,\C,\H,\O  \}$ and $n=2$ if $\mathbb{F}=\O$. Furthermore,  the metric of $\mathbb{F} \s H^n$ is such that its sectional curvature is equal to $-c$ in the real case, or pinched between $-c$ and $-c/4$ in the other cases, for some $c>0$. In this case we will write $\mathbb{F} \s H^n(c)$, and simply $\mathbb{F} \s H^n$ when $c=1$. Wolf \cite{wolfrankone} classified totally geodesic submanifolds in symmetric spaces of rank one and compact type.  Hence, by duality we  obtain the list of proper, non-flat, totally geodesic submanifolds of symmetric spaces of non-compact type and rank one up to congruence (see Table \ref{table:rankone}).
\begin{table}[h!]
	\renewcommand{\arraystretch}{1.2}
	\centering
		\begin{tabular}{lll}
	\hline
	\multicolumn{1}{|l}{$\mathbb{R}\s H^n$} &                                                          & \multicolumn{1}{l|}{}                  \\ \hline
	\multicolumn{1}{|c|}{}               & \multicolumn{1}{l|}{$\mathbb{R}\s H^k$}                     & \multicolumn{1}{c|}{$2\leq k\leq n-1$} \\ \hline
	\multicolumn{1}{|l}{$\mathbb{C}\s H^n$} &                                                          & \multicolumn{1}{l|}{}                  \\ \hline
	\multicolumn{1}{|l|}{}               & \multicolumn{1}{l|}{$\mathbb{C}\s H^k$}                     & \multicolumn{1}{c|}{$2\leq k\leq n-1$} \\
	\multicolumn{1}{|l|}{}               & \multicolumn{1}{l|}{$\mathbb{R}\s H^k(1/4)$}                     & \multicolumn{1}{c|}{$2\leq k\leq n$} \\
	\multicolumn{1}{|l|}{}               & \multicolumn{1}{l|}{$\mathbb{R}\s H^2$}                & \multicolumn{1}{l|}{ } \\ \hline
	\multicolumn{1}{|l}{$\mathbb{H}\s H^n$} &                                                          & \multicolumn{1}{l|}{}                  \\ \hline
	\multicolumn{1}{|l|}{}               & \multicolumn{1}{l|}{$\mathbb{H}\s H^k$}                     & \multicolumn{1}{c|}{$2\leq k\leq n-1$} \\
	\multicolumn{1}{|l|}{}               & \multicolumn{1}{l|}{$\mathbb{C}\s H^k$}                     & \multicolumn{1}{c|}{$2\leq k\leq n$} \\
	\multicolumn{1}{|l|}{}               & \multicolumn{1}{l|}{$\mathbb{R}\s H^k(1/4)$}                & \multicolumn{1}{c|}{$2\leq k\leq n$} \\
	\multicolumn{1}{|l|}{}               & \multicolumn{1}{l|}{$\mathbb{R}\s H^k$} & \multicolumn{1}{c|}{$2\leq k\leq 4$}   \\ \hline
	\multicolumn{1}{|l}{$\mathbb{O}\s H^2$} &                                                          & \multicolumn{1}{l|}{}                  \\ \hline
	\multicolumn{1}{|l|}{}               & \multicolumn{1}{l|}{$\mathbb{H}\s H^2$}                     & \multicolumn{1}{l|}{}  \\
	\multicolumn{1}{|l|}{}               & \multicolumn{1}{l|}{$\mathbb{C}\s H^2$}                     & \multicolumn{1}{l|}{}  \\
	\multicolumn{1}{|l|}{}               & \multicolumn{1}{l|}{$\mathbb{R}\s H^2(1/4)$}                & \multicolumn{1}{l|}{}                  \\
	\multicolumn{1}{|l|}{}               & \multicolumn{1}{l|}{$\mathbb{R}\s H^k$}                     & \multicolumn{1}{c|}{$2\leq k \leq 8$} \\ \hline
	&                                                          &                                       
\end{tabular}
	\vspace{0.2cm}
\caption{Totally geodesic submanifolds in symmetric spaces of non-compact type and rank one, up to congruence.}
\label{table:rankone}
\end{table}

Notice that in $\mathbb{C}\s H^2$ there are two non-congruent totally geodesic submanifolds homothetic to $\mathbb{R} \s H^2$:  $\mathbb{R} \s H^2(1/4)$ and  $\mathbb{R} \s H^2$ (which are totally real and complex in $\mathbb{C} \s H^2$, respectively). Also, in $\mathbb{H}\s H^3$ there are two non-congruent totally geodesic submanifolds homothetic to $\mathbb{R}\s H^3$: $\mathbb{R}\s H^3(1/4)$ and $\mathbb{R}\s H^3\subset\mathbb{R}\s H^4$. Finally, in $\mathbb{H}\s H^4$ there are two non-congruent totally geodesic submanifolds homothetic to $\mathbb{R}\s H^4$: $\mathbb{R}\s H^4(1/4)$ and $\mathbb{R}\s H^4$.

Now, we will set the following notation for the rest of this section. Let $M=M_1\times\cdots\times M_r$, where $M_i=\s G_i/\s K_i=\mathbb{F}_i\s H^{n_i}(c_i)$ is a symmetric space of non-compact type and rank one for each $i\in\{1,\ldots,r\}$.  Let $o=(o_1,\ldots,o_r)\in M$. Hence, we can identify $T_o M$ with a Lie triple system $\g{p}$ such that $\g{p}=\bigoplus_{i=1}^r \g{p}_i$, where $\g{p}_i$ is identified with $T_{o_{i}}\mathbb{F}_i \s H^{n_i}(c_i)$. This implies that $\g{g}_i=\g{p}_i\oplus[\g{p}_i,\g{p}_i]$ is the Lie algebra of $\s G_i$. 
\begin{lemma}
	\label{lemma:kblock}
	Let $M=M_1\times \cdots \times M_r$, where each $M_i$ is a simply connected, irreducible symmetric space. Let $\Sigma$ be an irreducible, non-flat, totally geodesic submanifold of $M$.
Then:
	\begin{enumerate}
		\item[i)] $\Sigma$ is $k$-diagonal for some $k\in\{1,\ldots,r\}$ in $M$. 
	\end{enumerate}
Moreover, if $M$ is of non-compact type, the following statements hold:
	\begin{enumerate}
		\item[ii)] There is some permutation $\sigma$ of $\{1,\ldots,r\}$ such that $\Sigma$ is a totally geodesic submanifold of $N_{\sigma(1)}\times \cdots \times N_{\sigma(k)}$, where $N_{\sigma(j)}:=\proj_{\sigma(j)} \Sigma$, $j\in\{1,\ldots,r\}$, is a totally geodesic submanifold of $M_{\sigma(j)}$ for every $j\in\{1,\ldots,k\}$.
		\item[iii)] The embedding of $\Sigma$ is given by $\Psi\colon \Sigma\rightarrow M$, $\Psi(p)=(\Psi_1(p),\ldots,\Psi_r(p))$, where each map $\Psi_j:=\proj_j\colon \Sigma\rightarrow N_{\sigma(j)}$ is a homothety for every $j\in\{1,\ldots,k\}$, and $\Psi_l\colon\Sigma\rightarrow N_{\sigma(l)}$ is a constant map for every $l\in\{k+1,\ldots,r\}$.
	\end{enumerate}
\end{lemma}
\begin{proof}
	Let $\Sigma$ be an irreducible, non-flat, totally geodesic submanifold of $M$. Let $\proj_i\colon \g{g}_{\Sigma}\rightarrow \g{g}_i$ be the $i$-th orthogonal projection onto $\g{g}_i$ for $i\in\{1,\ldots,r\}$. We will prove that $\proj_i$ is either the zero map or injective. Since $\Sigma$ is an irreducible symmetric space, its isotropy representation is irreducible. Furthermore, since $\Sigma$ is semisimple by our assumptions, we have that $\g{k}_{\Sigma}:=[\g{p}_{\Sigma},\g{p}_{\Sigma}]$ is the Lie algebra of the isotropy group of $\Sigma$. However, $\Ker \proj_{i{\vert\g{p}_{\Sigma}}}\subset \g{p}_{\Sigma}$ is an invariant subspace under the isotropy representation of $\Sigma$ since 
	\[\proj_{i}[Z,X]=[\proj_{i} Z, \proj_{i{\vert\g{p}_{\Sigma}}} X]=0,\]
	 where $X\in \Ker \proj_{i{\vert\g{p}_{\Sigma}}}$ and $Z\in \g{k}_{\Sigma}$. Thus, $\Ker  \proj_i =0$ or $\Ker \proj_{i{\vert\g{p}_{\Sigma}}}=\g{p}_{\Sigma}$. This implies  
	$\g{p}_{\Sigma}\subset \g{p}_{\sigma(1)}\oplus \cdots\oplus \g{p}_{\sigma(k)}$ for some $k\in\{1,\ldots,r\}$ and some permutation $\sigma$ of $\{1,\ldots,r\}$ such that $\Ker\proj_{\sigma(j){\vert\g{p}_{\Sigma}}}=0$ for $j\in\{1,\ldots,k\}$ and $\Ker \proj_{\sigma(l){\vert\g{p}_{\Sigma}}}=\g{p}_{\Sigma}$ for $l\in\{k+1,\ldots,r\}$. For a non-zero $X\in\g{p}_{\Sigma}$, we have $\proj_{\sigma(j)} X\neq 0$ if and only if $j\in\{1,\ldots, k\}$. Hence, every non-zero element $X$ in $\g{p}_{\Sigma}$ can be written as $X=\sum_{j=1}^k X_j$, where each $X_j\in \g{p}_{\sigma(j)}$ is non-zero. Thus, $\Sigma$ is $k$-diagonal and we proved \textit{i)}.

	Furthermore, $\g{p}_{\Sigma}\subset \proj_{\sigma(1)} \g{p}_{\Sigma}\oplus\cdots\oplus\proj_{\sigma(k)} \g{p}_{\Sigma}$, where each $\proj_{\sigma(j)} \g{p}_{\Sigma}$ is a Lie triple system of $\g{p}_{\Sigma}$ by Lemma \ref{lemma:projtriple}.
	  Moreover, $\Ker\proj_{\sigma(j)}$ is an ideal of $\g{g}_{\Sigma}$. For every $j\in\{1,\ldots,k\}$ we have $\Ker\proj_{{\sigma(j)\rvert \g{p}_{\Sigma} }}=0$, which implies that $\Ker\proj_{\sigma(j)}=0$, since $\g{g}_{\Sigma}$ is simple as $\Sigma$ is an irreducible symmetric space of non-compact type. Thus, $\proj_{\sigma(j)}\colon\g{g}_{\Sigma}\rightarrow\proj_{\sigma(j)}\g{g}_{\Sigma}$ is a Lie algebra isomorphism, for each $j\in\{1,\ldots,k\}$. Hence, $\proj_{\sigma(j)} \g{g}_{\Sigma}=\proj_{\sigma(j)}\g{p}_{\Sigma}\oplus[\proj_{\sigma(j)}\g{p}_{\Sigma},\proj_{\sigma(j)}\g{p}_{\Sigma}]$ is the Lie algebra of the isometry group of $N_{\sigma(j)}$, the totally geodesic submanifold of $M_{\sigma(j)}$ associated with the Lie triple system $\proj_{\sigma(j)}\g{p}_{\Sigma}$ in $\g{p}_{\sigma(j)}$.	Additionally, taking into account that $\Ker\proj_{\sigma(j)\vert_{\g{p}_{\Sigma}}}=\g{p}_{\Sigma}$ if and only if $l\in\{k+1,\ldots,r\}$, we have $\Sigma$ projects onto a point in $M_{\sigma(l)}$ if and only if $l\in\{k+1,\ldots,r\}$.   Therefore, we have that $\Sigma\subset N_{\sigma(1)}\times \cdots\times N_{\sigma(k)}$, where $N_{\sigma(j)}$ is a totally geodesic submanifold of  $M_{\sigma(j)}$, which proves \textit{ii)}.
	
 Finally, we will prove that $\proj_{\sigma(i)}$ is a homothety between $\Sigma$ and $N_{\sigma(i)}$. Let us fix some $i\in\{1,\ldots,k\}$  and let us consider the inner product in $\g{p}_{\Sigma}$ given by $(X,Y)_i:=\langle \proj_{\sigma(i)} X, \proj_{\sigma(i)} Y\rangle$, for $X, Y\in\g{p}_{\Sigma}$, where $\langle \cdot, \cdot \rangle$ is the inner product in $\g{p}$ induced by its identification with $T_o M$. Let $g\in \s  K_{\Sigma}$, where $\s K_{\Sigma}$ is the connected Lie subgroup of $\s G_{\Sigma}$ with Lie algebra $\g{k}_{\Sigma}$. Then, $g=\prod_{i=1}^k g_{\sigma(i)}$ for some $g_{\sigma(i)}\in \pi_i \s K_{\Sigma}$, where $\pi_i\colon \s G_1\times\cdots\times \s G_r\rightarrow \s G_i$ is the  projection onto the $i$-th factor. Thus, for any $X,Y\in\g{p}_{\Sigma}$,
	\begin{align*}
	(\Ad(g) X, \Ad(g)Y)_i&=\langle \proj_{\sigma(i)} \Ad(g) X,\proj_{\sigma(i)} \Ad(g) Y\rangle\\
	  &=\langle\proj_{\sigma(i)} \Ad(g_{\sigma(i)}) X,\proj_{\sigma(i)} \Ad(g_{\sigma(i)}) Y  \rangle\\
	  &=\langle  \Ad(g_{\sigma(i)})\proj_{\sigma(i)} X, \Ad(g_{\sigma(i)})\proj_{\sigma(i)}  Y  \rangle\\
	  &=\langle \proj_{\sigma(i)} X,\proj_{\sigma(i)} Y\rangle=(X,Y)_i,
	\end{align*}
	where we have used that $\Ad(g_{\sigma(i)})$ is a linear isometry for $\langle\cdot,\cdot\rangle$ which leaves $\proj_{\sigma(i)}\g{p}_{\Sigma}$ invariant, since $g_{\sigma(i)}$ belongs to the isotropy of $N_{\sigma(i)}$ for each $i\in\{1,\ldots,k\}$.
	Hence, $(\cdot,\cdot)_i$ is a $\s K_{\Sigma}$-invariant inner product in $\g{p}_{\Sigma}$. Moreover, by Schur Lemma, since the isotropy representation of $\Sigma$ is irreducible by assumption, we have that  $\proj_{\sigma(i)}$ is a homothety between the Lie triple systems $\g{p}_{\Sigma}$ and $\proj_{\sigma(i)}\g{p}_{\Sigma}$ for each $i\in\{1,\ldots,k\}$. In addition to that, $\proj_{\sigma(i)}\colon \g{p}_{\Sigma}\rightarrow\proj_{\sigma(i)}\g{p}_{\Sigma}$, preserves the sectional curvature since it preserves the Lie bracket. Thus, by \cite[Theorem 1.9.2]{wolf}, we have that $\proj_{\sigma(i)}\colon \Sigma\rightarrow\proj_{\sigma(i)}\Sigma$ is an affine diffeomorphism since $\Sigma$ and $\proj_{\sigma(i)}\Sigma$ are simply connected. Now let $p\in \Sigma$, $\gamma$ be a path in $\Sigma$ joining $o$ and $p$ and $\widetilde{\gamma}:=\proj_{\sigma(i)}\gamma$. Let $P_{\gamma}$ and $P_{\widetilde{\gamma}}$ be the parallel transports along to $\gamma$ and $\widetilde{\gamma}$, respectively. Since $\proj_{\sigma(i)}$ is affine, we have $\proj_{{\sigma(i)}_{*p}}=P_{\widetilde{\gamma}}\circ \proj_{{\sigma(i)}_{*o}}\circ P^{-1}_{\gamma}$ for every $p\in M$. However, $P_{\gamma}$ and $P_{\widetilde{\gamma}}$ are isometries and $\proj_{{\sigma(i)}_{*o}}$ is a homothety since $\proj_{\sigma(i)}\colon \g{p}_{\Sigma}\rightarrow \proj_{\sigma(i)}\g{p}_{\Sigma}$ is a homothety. Thus, $\proj_{{\sigma(i)}_{*p}}$ is also a homothety for every $p\in M$ and it turns out that $\proj_i\colon \Sigma\rightarrow \proj_{\sigma(i)}\Sigma$ is a homothety. Consequently, we have proved \textit{iii)}.\qedhere
\end{proof} 
\begin{remark}
	\label{remark:homothety}
	Observe that this lemma admits a converse. Let $M=M_1\times\cdots\times M_r$, where each $M_i$ is an irreducible symmetric space of non-compact type. Let $\Sigma$ be a Riemannian manifold and consider the embedding $\Psi\colon \Sigma\rightarrow M$, $\Psi(p)=(\Psi_1(p),\ldots,\Psi_r(p))$, where each map $\Psi_j\colon\Sigma\rightarrow N_{j}$ is either a homothety or a constant map, and $N_{j}$ is any totally geodesic submanifold of $M_{j}$ where $j\in\{1,\ldots,r\}$. Then, $\Psi(\Sigma)$ is a totally geodesic submanifold of $N_{1}\times\cdots\times N_{r}$, since homotheties carry geodesics into geodesics. Therefore, $\Psi(\Sigma)$ is a totally geodesic submanifold of $M$. 
\end{remark}
\begin{remark}
\label{remark:k-diagonaltotgeod}
One important consequence of the previous result that deserves to be highlighted is the following. Let $\Sigma$ be an irreducible, non-flat, $r$-diagonal totally geodesic submanifold of $M=M_1\times\cdots \times M_r$, where $M_i$ is an irreducible symmetric space of non-compact type for each $i\in\{1,\ldots,r\}$. Then, with the usual notation, $\g{p}_{\Sigma}\subset \g{p}_1\oplus\cdots\oplus\g{p}_r$ is an $r$-diagonal Lie triple system, and we can define a  Lie algebra isomorphism $\Phi_i:=\proj_i\colon \g{g}_{\Sigma} \rightarrow \proj_i\g{g}_\Sigma$, which sends $\g{p}_{\Sigma}$ onto the Lie triple system $\proj_i\g{p}_\Sigma$ in $\g{p}_i$.
Therefore, $\g{p}_{\Sigma}=\{ \sum_{i=1}^r \varphi_i X: X\in \proj_1\g{p}_\Sigma\}$, where $\varphi_i:=\Phi_i\Phi^{-1}_1\colon \proj_1\g{g}_\Sigma\rightarrow \proj_i\g{g}_\Sigma$ is a Lie algebra isomorphism sending $\proj_1\g{p}_\Sigma$ onto $\proj_i\g{p}_\Sigma$ for each $i\in\{1,\ldots,r\}$. Let $s\in\{1,\ldots,r\}$. Notice that $\g{p}_{\Sigma^s}=\{\sum_{i=1}^s \varphi_i X: X\in\proj_1\g{p}_\Sigma \}$ is an $s$-diagonal Lie triple system in $\g{p}$ such that $\g{g}_{\Sigma^s}$ is isomorphic to $\g{g}_{\Sigma}$ and then $\Sigma^s$, the  totally geodesic submanifold of $M$ corresponding to $\g{p}_{\Sigma^s}$, is homothetic to $\Sigma$.
\end{remark}

\begin{proposition}
	\label{proposition:compactnoncompact}
	Let $M_1$ and $M_2$ be irreducible symmetric spaces of compact and non-compact type, respectively. If  $\Sigma\subset M_1\times M_2$ is a diagonal totally geodesic submanifold, then  $\Sigma$ is flat.
\end{proposition}
\begin{proof}
	Let $\Sigma$ be a totally geodesic submanifold of $M_1\times M_2$. By De-Rham Theorem, we have that the universal covering of $\Sigma$ is $\widetilde{\Sigma}=\Sigma_0\times\Sigma_1\times\cdots\times\Sigma_s$, where $\Sigma_0$ is flat and each $\Sigma_i$ is an irreducible semisimple symmetric space. Moreover, $\g{p}_{\Sigma}=\bigoplus_{i=0}^s
	\g{p}_{\Sigma_i}$, where $\g{p}_{\Sigma_i}\subset \g{p}_{\Sigma} \subset\g{p}_1\oplus\g{p}_2$ is a Lie triple system corresponding to the irreducible symmetric space $\Sigma_i$, for each $i\in\{0,\ldots,s\}$, and $\g{p}_1,\g{p}_2$ are the Lie triple systems corresponding to $M_1$ and $M_2$, respectively. 
	
	Let us fix some $i\in\{1,\ldots,s\}$. Then  $\Sigma_i$ is semisimple, and we have that $\g{k}_{\Sigma_i}:=[\g{p}_{\Sigma_i},\g{p}_{\Sigma_i}]$ is the Lie algebra of the isotropy of $\Sigma_i$, and $\g{g}_{\Sigma_i}:=\g{k}_{\Sigma_i}\oplus\g{p}_{\Sigma_i}$ is the Lie algebra of the isometry group of $\Sigma_i$. Now, we define $\varphi_{i j}\colon \g{g}_{\Sigma_i}\rightarrow\g{g}_j$, where $\varphi_{i j} X=\proj_j X$ for each $i\in\{1,\ldots,s\}$ and $j\in\{1,2\}$, $\g{g}_j$ is the Lie algebra of the isometry group of $M_j$, and $\proj_j\colon \g{g}_1\oplus\g{g}_2\rightarrow \g{g}_j$ is the projection map.
	Notice that $\Ker \varphi_{ij\vert\g{p}_{\Sigma_i}}\subset \g{p}_{\Sigma_i}$ is an invariant subspace for the isotropy representation of $\Sigma_i$.  Since $\Sigma_i$ is irreducible, we have $\Ker \varphi_{ij\vert\g{p}_{\Sigma_i}}=0$ or $\Ker \varphi_{ij\vert\g{p}_{\Sigma_i}}=\g{p}_{\Sigma_i}$. Moreover, as $\g{p}_{\Sigma_i}$ is diagonal by assumption, $\Ker \varphi_{ij\vert\g{p}_{\Sigma_i}}=0$ for every $j\in\{1,2\}$. On the one hand, if $\Sigma_i$ is a compact simple Lie group, $\g{k}_{\Sigma_i}$ is simple, and hence, $\Ker\varphi_{ij\vert\g{k}_{\Sigma_i}}=0$ or $\Ker\varphi_{ij\vert\g{k}_{\Sigma_i}}=\g{k}_{\Sigma_i}$, since $\Ker\varphi_{ij\vert\g{k}_{\Sigma_i}}$ is an ideal of $\g{k}_{\Sigma_i}$. In any other case, $\g{g}_{\Sigma_i}$ is simple. Then, we have $\Ker\varphi_{ij}=0$ or $\Ker\varphi_{ij}=\g{g}_{\Sigma_i}$. However, $\Ker\varphi_{ij\vert\g{p}_{\Sigma_i}}=0$ and this implies that $\Ker\varphi_{ij}=0$.
	
	To sum up, for each $j\in\{1,2\}$, we have  $\Ker\varphi_{ij}=0$ or $\Ker\varphi_{ij}=\g{k}_{\Sigma_i}$. Let us assume that $\Ker\varphi_{ij}=\g{k}_{\Sigma_i}$ for some $j\in\{1,2\}$. In this case $\varphi_{ij}\g{g}_{\Sigma_i}=\varphi_{ij}\g{p}_{\Sigma_i}$ is an abelian Lie algebra. Since $\g{g}_{\Sigma_i}$ is semisimple and $\g{g}_{\Sigma_i}\subset \varphi_{ij}\g{g}_{\Sigma_i}\oplus\varphi_{ik}\g{g}_{\Sigma_i}$, where $k\in\{1,2\}\setminus\{j\}$,  we have  
	\[\g{g}_{\Sigma_i}=[\g{g}_{\Sigma_i},\g{g}_{\Sigma_i}]\subset[\varphi_{ij}\g{g}_{\Sigma_i}\oplus\varphi_{ik}\g{g}_{\Sigma_i},\varphi_{ij}\g{g}_{\Sigma_i}\oplus\varphi_{ik}\g{g}_{\Sigma_i}]\subset[\varphi_{ik}\g{g}_{\Sigma_i},\varphi_{ik}\g{g}_{\Sigma_i}]\subset\varphi_{ik}\g{g}_{\Sigma_i}\subset \g{g}_k.\]
	Therefore, we obtain a contradiction with the assumption that $\g{p}_{\Sigma}$ is diagonal.
	Now let us assume that $\Ker\varphi_{ij}=0$ for every $j\in\{1,2\}$. This implies that $\g{g}_{\Sigma_i}$ and $\varphi_{ij}\g{g}_{\Sigma_i}$ are isomorphic  for every $j\in\{1,2\}$. In particular, $\varphi_{ij}\g{g}_{\Sigma_i}$ is simple for every $i\in\{1,\ldots,s\}$ and $j\in\{1,2\}$, and $\varphi_{i1}\g{g}_{\Sigma_i}$ is isomorphic to $\varphi_{i2}\g{g}_{\Sigma_i}$. Now,  as $\varphi_{i1}\g{g}_{\Sigma_i}$ is a subalgebra of $\g{g}_1$, we have that $\varphi_{i1}\g{g}_{\Sigma_i}$ is a compact Lie algebra. Moreover, $\varphi_{i2}\g{g}_{\Sigma_i}=\varphi_{i2}\g{p}_{\Sigma_i}\oplus[\varphi_{i2}\g{p}_{\Sigma_i},\varphi_{i2}\g{p}_{\Sigma_i}]$, where $\varphi_{i2}\g{p}_{\Sigma_i}$ is a Lie triple system, is not a compact Lie algebra since it is simple and then it is the Lie algebra of the isometry group of an irreducible symmetric space of non-compact type (see discussion above Lemma~\ref{lemma:kblock}). 
	
	Consequently, we obtain a contradiction with the existence of (non-trivial) irreducible semisimple factors of $\Sigma$, which yields our result.
\end{proof}

\begin{proposition}
	\label{prop:congruency}
	Let $\Sigma_1,\Sigma_2$ be $r$-diagonal, non-flat, irreducible, totally geodesic submanifolds in $M=M_1\times\cdots\times M_r$, where each $M_i$ is an irreducible symmetric space of non-compact type homothetic to both $\Sigma_1$ and $\Sigma_2$. Then, there is $g\in\Isom(M_1)\times\cdots\times\Isom(M_r)\leq\Isom(M)$ such that $g\Sigma_1=\Sigma_2$.
\end{proposition}
\begin{proof}
	Let $\g{p}_{\Sigma_1}=\{ \sum_{i=1}^{k} \varphi_i X : X\in \g{p}_1  \}$, $\g{p}_{\Sigma_2}=\{ \sum_{i=1}^{k} \psi_i X : X\in \g{p}_1   \}$ and $\g{g}_{\Sigma_j}:=\g{p}_{\Sigma_j}\oplus[\g{p}_{\Sigma_j},\g{p}_{\Sigma_j}]$, for $j\in\{1,2\}$, where $\varphi_i,\psi_i\colon\g{g}_1\rightarrow\g{g}_i$ are Lie algebra isomorphisms sending $\g{p}_1$ onto $\g{p}_i$ for $i\in\{2,\ldots,r\}$, and $\varphi_1=\psi_1=\Id_{\g{g}_1}$, where $\Id_{\g{g}_1}$ is the identity map of $\g{g}_1$. Here we are using Remark \ref{remark:k-diagonaltotgeod} along with the assumption that each $M_i$ is an irreducible symmetric space of non-compact type homothetic to $\Sigma_1$ and $\Sigma_2$.
	
	Let  $\sigma_i:=\psi_i\varphi^{-1}_i\in \Aut(\g{g}_i)$. First of all, observe that $\sigma_{i\vert_{\g{p}_i}}$ is a linear isometry of $\g{p}_i$ since $\sigma_i\in\Aut(\g{g}_i)$ and the inner product on $\g{p}_i$ is the restriction of the Killing form  of $\g{g}_i$, up to scaling. Furthermore, $\sigma_{i\vert_{\g{p}_i}}$ preserves the curvature tensor of $M_i$ at $o_i$ since this is given by Lie brackets. Hence, $\sigma_{i\vert_{\g{p}_i}}$ is a linear isometry of $\g{p}_i$ that preserves sectional curvature at $o_i$. Thus, by \cite[Corollary 2.3.14]{wolf}, $\sigma_{i\vert_{\g{p}_i}}$ extends to an isometry $g_i\in \Isom(M_i)$ that fixes $o_i\in M_i$, since it leaves $\g{p}_i$ invariant. Then, $\sigma_i=\Ad(g_i)$ and if $g:=\prod_{i=1}^r g_i$, where $g_1$ is the identity element of $\Isom(M_1)$, we obtain
	\begin{align*}
	\Ad(g)^{-1}\g{g}_{\Sigma_2}&=\prod_{i=1}^r\Ad(g_i)^{-1}\g{g}_{\Sigma_2}=\bigg\{\sum_{i=1}^r\Ad(g_i)^{-1}\psi_i X: X\in \g{g}_1\bigg\}=\bigg\{\sum_{i=1}^r \varphi_i X: X\in \g{g}_1\bigg\}\\
	&=\g{g}_{\Sigma_1},
	\end{align*}
	where we have used  $\sigma_i=\psi_i\varphi^{-1}_i$ and  $\Ad(g_i)_{|\mathfrak{g}_j}=\Id_{\mathfrak{g}_j}$ for $i\neq j$.
	Therefore, there is a $g\in\Isom(M_1)\times\cdots\times\Isom(M_r)\leq\Isom(M)$ such that $g \s G_{\Sigma_1}= \s G_{\Sigma_2} g$. Consequently, $g \Sigma_1=g \s G_{\Sigma_1}\cdot o=\s G_{\Sigma_2} g\cdot o=\s G_{\Sigma_2} \cdot o=\Sigma_2$, since $g$ fixes $o$.\qedhere
\end{proof}

\begin{lemma}
	\label{lemma:proj}
	Let $\Sigma$ be a totally geodesic submanifold of $ M=M_1\times \cdots \times M_r$, where $M_i$ is a symmetric space of non-compact type and rank one, for each $i\in\{1,\ldots,r\}$. Moreover, let $\Sigma=\Sigma_1\times\Sigma_2$, where $\Sigma_1$ is irreducible and not flat. Then, if $\proj_i \Sigma_1$ and $\proj_j \Sigma_2$ have positive dimension, we have that $i\neq j$.
\end{lemma}
\begin{proof}
	Let $\Sigma=\Sigma_1\times\Sigma_2$ be a totally geodesic submanifold of $M$, where $\Sigma_1$ is irreducible and not flat.
	Let $\g{p}_{\Sigma}=\g{p}_{\Sigma_1}\oplus\g{p}_{\Sigma_2}\subset \g{p}=\bigoplus_{i=1}^r \g{p}_i$ be the corresponding Lie triple system. Fix some $i\in\{1,\ldots,r\}$. Let us suppose that $\widehat{\g{p}}_j:=\proj_i \g{p}_{\Sigma_j}$ has positive dimension for both $j=1$ and $j=2$. By Lemma \ref{lemma:kblock}, we have that $\dim \widehat{\g{p}}_1>1$, since $\Sigma_1$ is irreducible and not flat. Thus, we can choose $X\in\widehat{\g{p}}_1$ and $Y\in \widehat{\g{p}}_2$ spanning a 2-plane in $\g{p}_i$. Moreover, the sectional curvature $\sec$ of $M_i$ is given by
	\[  \sec(X,Y)=-\frac{\langle [[X,Y],Y],X\rangle}{\langle X,X\rangle\langle Y,Y\rangle- \langle X,Y\rangle^2}.   \]
	In particular,  $[\widehat{\g{p}}_1,\widehat{\g{p}}_2]\neq 0$, since $M_i$ has negative sectional curvature. However, we have  
	$[\widehat{\g{p}}_1,\widehat{\g{p}}_2]=[\proj_i \g{p}_{\Sigma_1},\proj_i\g{p}_{\Sigma_2}]=\proj_i[\g{p}_{\Sigma_1},\g{p}_{\Sigma_2}]=0$, since $\Sigma$ is a Riemannian product of the symmetric spaces $\Sigma_1$ and $\Sigma_2$. Therefore, we obtain a contradiction with the assumption that both $\widehat{\g{p}}_1$ and $\widehat{\g{p}}_2$ have positive dimension.\qedhere
\end{proof}
\begin{remark}
\label{remark:bigrk}
It is important to notice that the previous lemma is not true when the ambient space is a product of irreducible symmetric spaces of rank greater than one. For instance, one can find a totally geodesic submanifold $\Sigma$ homothetic to $\R \s H^2\times \R \s H^2$ in $M=M_1\times M_2$, with $M_i=\mathsf{SO}_{2,4}/(\mathsf{SO}_2\times \mathsf{SO}_4)$ for each $i\in\{1,2\}$, such that both factors of $\Sigma$ have non-trivial projection onto each factor of $M$. Thus, applying \cite[Theorem 4.1 and \S 5]{kleindga} and duality, there is a totally geodesic submanifold $\Sigma^1_i\times\Sigma^2_i\subset M_i$, where $\Sigma^1_i$ and $\Sigma^2_i$ are mutually isometric real hyperbolic planes, for each $i\in\{1,2\}$. Now we can consider $\widehat{\Sigma}^j$ a $2$-diagonal totally geodesic submanifold in $\Sigma^j_1\times\Sigma^j_2$ for each $j\in\{1,2\}$. Thus, $\Sigma:=\widehat{\Sigma}^1\times\widehat{\Sigma}^2$ is a totally geodesic submanifold homothetic to $\R\s H^2\times\R \s H^2$ in $M$ such that both factors of $\Sigma$ have non-trivial projection onto each factor of $M$.
\end{remark}
We now introduce a notation that will be useful in what follows. Let $M$ and $\Sigma$ be two symmetric spaces. We will write  $(\Sigma)\leq M$ if $M$ contains a totally geodesic submanifold isometric to $\Sigma$.

A simply connected, reducible symmetric space $M$ of rank 2 is a product of two simply connected irreducible symmetric spaces of rank one, $M_1$ and $M_2$. Let $\Sigma$ be a totally geodesic submanifold in $M=M_1\times M_2$. Notice that if $\Sigma$ is reducible, then it has maximal rank and Proposition \ref{prop:maximalrank} implies that $\Sigma=\Sigma_1\times\Sigma_2$, where $\Sigma_i$ is a totally geodesic submanifold of $M_i$ for each $i\in\{1,2\}$. Now, if $\Sigma$ is irreducible, it must be either a geodesic or a semisimple totally geodesic submanifold. Let us assume that $\Sigma$ is an irreducible semisimple totally geodesic submanifold of $M$. Then, by Lemma \ref{lemma:kblock}~\textit{i}), $\Sigma$ is either $1$-diagonal or $2$-diagonal. If $\Sigma$ is $1$-diagonal,  clearly, $\Sigma=\Sigma_i\times \{p_j\}$, where $\Sigma_i$ is a totally geodesic submanifold of $M_i$ and $p_j\in M_j$ for distinct $i,j\in\{1,2\}$. 

Let us assume that $M_1$ is of compact type and $M_2$ is of non-compact type. By Proposition~\ref{proposition:compactnoncompact}, $M=M_1\times M_2$ has no  diagonal totally geodesic submanifolds of dimension greater than one. Hence, a $2$-diagonal totally geodesic submanifold is a geodesic. 

Let us assume that $M_1$ is flat and $M_2$ is of non-compact type.  We can suppose that $\Sigma$ is an irreducible semisimple $2$-diagonal totally geodesic submanifold. Thus, every non-zero vector in $\g{p}_{\Sigma}$ is of the form $X=X_1+X_2$, where $X_i\in\g{p}_i$ is a non-zero vector in $\g{p}_i$ for each $i\in\{1,2\}$. However, since $\Sigma$ is semisimple, then $\g{p}_{\Sigma}$ has dimension greater than one. Moreover, $\dim\proj_1\g{p}_\Sigma=1$. Hence, there exists some non-zero vector $X'$ in $\g{p}_{\Sigma}\cap\proj_2\g{p}_\Sigma$  contradicting the assumption that $\Sigma$ is $2$-diagonal.

To sum up the above discussion: if $M_1$ and $M_2$ have opposite types or one of them is flat, then every totally geodesic submanifold $\Sigma$ in $M=M_1\times M_2$ is either a geodesic or  equal to $\Sigma_1\times\Sigma_2$, where $\Sigma_i\subset M_i$ is a totally geodesic submanifold for each $i\in\{1,2\}$. In view of the argumentation above, by duality we will assume that both factors in $M$ are either of compact or of non-compact type. In this case, we have the following result.
\begin{theorem}
\label{th:rank2reducible}
Let $M_i:=\mathbb{F}_i\s H^{n_i}(c_i)$ be a symmetric space of non-compact type and rank one for $i=1,2$. Given positive numbers $c'_1,c'_2$, we define the quantity $c=\frac{c'_1c'_2}{c'_1+c'_2}$. Then,  $\Sigma\subset M_1 \times M_2$ is a  totally geodesic submanifold if and only if it is equal to one in the list below:
\begin{enumerate}[i)]
\item A geodesic in $M_1\times M_2$.	
\item A product $\Sigma_1\times \Sigma_2$, where $\Sigma_i\subset  M_i$ is a totally geodesic submanifold for $i\in\{1,2\}$.
\item A totally geodesic diagonal  $\mathbb{F}\s H^n(c)$, with  $\mathbb{F}\neq \mathbb{R}$ and $c'_i=c_i$,   whenever $(\mathbb{F}\s H^n(c_i))\leq M_i$ for every $i\in\{1,2\}$.
\item A totally geodesic diagonal $\mathbb{R}\s H^n(c)$, with  $c'_i~\in~\{c_i,\frac{c_i}{4}\}$, whenever $(\R \s H^n (c'_i))\leq M_i$ for every $i\in\{1,2\}$.
\end{enumerate}
\end{theorem}
\begin{remark}
The diagonal embeddings in items \textit{iii)} and \textit{iv)} are described in Lemma \ref{lemma:kblock}~\textit{iii)}. These are of the form $\Psi\colon \Sigma\rightarrow M_1\times M_2$, $p\in \Sigma\mapsto(\Psi_1(p),\Psi_2(p))$, where $\Psi_i$ is a  homothety between $\Sigma$ and some totally geodesic submanifold $N_i$ of $M_i$ for each $i\in\{1,2\}$ (see Remark~\ref{remark:homothety}).
\end{remark}
\begin{proof}
Let us assume that $\Sigma$ has rank two. Then, by Proposition~\ref{prop:maximalrank},  $\Sigma=\Sigma_1\times\Sigma_2$, where $\Sigma_i\subset M_i$ is totally geodesic for each $i=1,2$, which corresponds to item \textit{ii)} in the statement.

Now assume that $\Sigma$ has rank one, then it is either a geodesic, which corresponds to item \textit{i)}, or it is semisimple. In this latter case, $\Sigma$ must be isometric to $\mathbb{F}\s H^n(c)$ for some $\mathbb{F}\in\{\R,\C,\H,\O\}$ and $c>0$. If $\Sigma$ is not diagonal, it is 1-diagonal by Lemma \ref{lemma:kblock} and it must be congruent to $\Sigma_1 \times \{p_2\}$ or to $\{p_1\}\times \Sigma_2$, where $\Sigma_i\subset \mathbb{F}_i\s H^{n_i}(c_i)$ is a totally geodesic submanifold and $p_i\in M_i$ for $i=1,2$. This corresponds to item \textit{ii)} in the statement.
Moreover, if it is diagonal, it is $2$-diagonal. Then,  by Lemma \ref{lemma:kblock} and the classification of totally geodesic submanifolds in the rank one symmetric spaces (see Table \ref{table:rankone}), we have that $\Sigma\subset \mathbb{F}\s H^{n}(c'_1)\times \mathbb{F} \s H^{n}(c'_2)$, for some positive numbers $c'_1$ and $c'_2$.
Let us further assume that $\mathbb{F}\neq \R$. Hence, by the classification in rank one, $c'_i=c_i$ for $i=1,2$. We will prove that the sectional curvature of $\Sigma$ satisfies
$$ \sec(X,Y)\in\left[-\frac{c_1 c_2}{c_1+c_2},-\frac{c_1 c_2}{4(c_1+c_2)} \right],$$
for any $X,Y\in\g{p}_{\Sigma}$ spanning a $2$-plane.
Let  $\g{p}'_i:=\proj_i\g{p}_{\Sigma}$ be a Lie triple system associated with $T_{o_i}\mathbb{F}\s H^{n}(c_i)$, where $o_i\in \mathbb{F}\s H^{n}(c_i)$. Moreover, consider  the Lie algebra of the isometry group of $\mathbb{F}\s H^{n}(c_i)$, which is  $\g{g}'_i:=\g{p}'_i\oplus[\g{p}'_i,\g{p}'_i]$. Then, the Lie triple system corresponding to $\Sigma$ is $\g{p}_{\Sigma}=\{X_1 +\varphi X_1\colon X_1\in \g{p}'_1\}$ for some Lie algebra isomorphism $\varphi\colon \g{g}'_1\rightarrow \g{g}'_2$ that sends $\g{p}'_1$ onto $\g{p}'_2$ (see Remark \ref{remark:k-diagonaltotgeod}). Now let $\langle \cdot,\cdot\rangle_1$ be the  metric of $\mathbb{F}\s H^{n}(1)$, and $\sec_1(\cdot,\cdot)$ its sectional curvature. We can regard the induced metric of $\mathbb{F}\s H^{n}(c_i)$ on $\g{p}'_i$ as a positive multiple of $\langle \cdot,\cdot \rangle_1$. Hence, we can write the metric of $\Sigma$ at $(o_1,o_2)$ as
\[ \langle \cdot, \cdot\rangle :=\lambda_1 \langle \proj_1 \cdot, \proj_1 \cdot \rangle_1 + \lambda_2 \langle \proj_2 \cdot, \proj_2 \cdot \rangle_1,   \]
for some $\lambda_1,\lambda_2>0$.
Let $X=X_1 +\varphi X_1,Y=Y_1+\varphi Y_1\in \g{p}_{\Sigma}$, where $X_1,Y_1\in\g{p}'_1$ satisfy $\langle X_1,X_1\rangle_1=\langle Y_1, Y_1\rangle_1=1$ and $\langle X_1, Y_1\rangle_1=0$. Moreover, we have
\begin{align*}
\langle [[X,Y],Y],X\rangle&=\lambda_1 \langle [[X_1,Y_1],Y_1],X_1\rangle_1 + \lambda_2 \langle [[\varphi X_1,\varphi Y_1],\varphi Y_1],\varphi X_1\rangle_1\\
&=\lambda_1 \langle [[X_1,Y_1],Y_1],X_1\rangle_1 + \lambda_2 \langle \varphi[[X_1,Y_1],Y_1],\varphi X_1\rangle_1\\
&=-(\lambda_1+\lambda_2) \sec_1(X_1,Y_1),\\
\langle X, X\rangle&=\langle X_1 +\varphi X_1, X_1 +\varphi X_1\rangle=\lambda_1\langle X_1, X_1\rangle_1 + \lambda_2 \langle \varphi X_1,\varphi X_1\rangle_1=\lambda_1+\lambda_2,\\
\langle Y, Y\rangle&=\langle Y_1 +\varphi Y_1, Y_1 +\varphi Y_1\rangle=\lambda_1\langle Y_1, Y_1\rangle_1 + \lambda_2 \langle \varphi Y_1,\varphi Y_1\rangle_1=\lambda_1+\lambda_2,\\
\langle X, Y\rangle&=\langle X_1 +\varphi X_1, Y_1 +\varphi Y_1\rangle=\lambda_1\langle X_1, Y_1\rangle_1 + \lambda_2 \langle \varphi X_1,\varphi Y_1\rangle_1=0,
\end{align*} 
where we have used that $\varphi$ preserves $\langle\cdot,\cdot\rangle_1$ since $\varphi$ preserves the Killing form.
Thus, the sectional curvature at $(o_1,o_2)\in\Sigma$ of the 2-plane spanned by $\{X,Y\}$ is given by 
\[\sec(X,Y)=\frac{\sec_1(X_1,Y_1)}{\lambda_1+\lambda_2}=\frac{c_1c_2}{c_1+c_2}\sec_1(X_1,Y_1)\in \left[-\frac{c_1 c_2}{c_1+c_2},-\frac{c_1 c_2}{4(c_1+c_2)} \right],  \]
since $\lambda_i=1/c_i$ for each $i\in\{1,2\}$, because  $\lambda g$ has sectional curvature $\frac{1}{\lambda}\sec$, when $g$ is a Riemannian metric, $\sec$ its sectional curvature and $\lambda$ a positive number. Hence, $\Sigma$ must be isometric to a diagonal $\mathbb{F}\s H^n(\frac{c_1 c_2}{c_1+c_2})$,  whenever $(\mathbb{F}\s  H^n(c_i))\leq M_i$ for every $i\in\{1,2\}$ and $\mathbb{F}\neq \mathbb{R}$. This corresponds to item \textit{iii)} in the statement.

Now let us assume that $\mathbb{F}=\R$. Again, by Lemma \ref{lemma:kblock} and the classification of totally geodesic submanifolds in symmetric spaces of non-compact type and rank one (see Table~\ref{table:rankone}), we have $\Sigma\subset \mathbb{R}\s H^{n}(c'_1)\times \mathbb{R}\s H^{n}(c'_2)$, where  $c'_i\in\{c_i,c_i/4\}$ is such that $(\R\s H^{n}(c'_i))\leq M_i$ for every $i\in\{1,2\}$. A similar computation as above yields that the sectional curvature of $\Sigma$ is equal to $-\frac{c'_1 c'_2}{c'_1+c'_2}$. Then, $\Sigma$ is isometric to $\R\s  H^n(\frac{c'_1 c'_2}{c'_1+c'_2})$, which corresponds to item \textit{iv)} in the statement.  \qedhere

\end{proof}
\begin{remark}
	Notice that unlike in the irreducible rank one case, if $M$ is a reducible symmetric space of rank two then we can find mutually isometric diagonal totally geodesic submanifolds  $\Sigma_1$ and $\Sigma_2$ which are not congruent in $M$. This implies that the  hypothesis of $M_i$ being homothetic to $\Sigma_1$ and $\Sigma_2$ for each $i\in\{1,2\}$ in Proposition \ref{prop:congruency} is crucial. Let us consider $M=M_1\times M_2$, where $M_1=\C \s H^2$ and $M_2=\C\s  H^3$. Let us assume that there is some $\varphi\in\Isom(M)$ such that $\varphi \Sigma_1=\Sigma_2$ where
	\begin{align*}
	\Sigma_1&:=\R \s H^2(1/5)\subset L_1:=\R \s H^2 \times \R \s H^2(1/4)\subset \C \s H^2\times \C \s H^2\subset M_1\times M_2,\\
	\Sigma_2&:=\R \s H^2(1/5)\subset L_2:=\R \s H^2(1/4) \times \R \s H^2\subset \C \s H^2\times \C \s H^2\subset M_1\times M_2.
	\end{align*}
	Clearly, each isometry of $M_1\times M_2$ must preserve both factors since $M_1$ and $M_2$ are not isometric. Then, we have $\Sigma_2\subset \varphi L_1\cap L_2$. However, since $\R \s H^2$ and $\R \s H^2(1/4)$ are complex and totally real submanifolds in $\C \s H^2$, respectively, we have  $\varphi L_1\cap L_2\subset \R\times \R$, where $\R\times \R$ is a maximal totally geodesic flat submanifold of $M_1\times M_2$. Moreover, since the intersection of totally geodesic submanifolds is totally geodesic, this implies that $\Sigma_2\subset \varphi L_1\cap L_2\subset \R\times\R\subset M_1\times M_2$, which contradicts the fact that $\Sigma_2$ is not flat and proves that such $\varphi$ cannot exist.
\end{remark}

Let us recall that  the elementary symmetric polynomial $e_k$  of order $k\in\{0,\ldots,n\}$ in $n$ variables  is defined as $e_k(X_1,\ldots,X_n)=\sum_{1\leq i_1<\ldots<i_k\leq n} X_{i_1}\cdots X_{i_{k}}$. Then, we have the following generalization of Theorem \ref{th:rank2reducible} to the case of diagonal totally geodesic submanifolds in arbitrary products of symmetric spaces of rank one.
\begin{corollary}
\label{cor:sympolynomial}
Let $M= M_1\times\cdots\times M_r$, where each $ M_i=\mathbb{F}_i \s H^{n_i}(c_i)$ is a symmetric space of non-compact type and rank one for $i\in\{1,\ldots,r\}$. Given positive numbers $\{c'_i\}_{i=1}^r$, we define the quantity
\[c:=\frac{\prod_{i=1}^r c'_{i}}{e_{r-1}(c'_{1},\ldots,c'_{r})},   \]
where $e_{r-1}$ is the elementary symmetric polynomial of degree $r-1$ in $r$ variables.

Then, if  $\Sigma$ is a non-flat, irreducible, $r$-diagonal, totally geodesic submanifold in $M$, it is isometric to one in the list below:
\begin{enumerate}[i)]
	\item $\R \s H^n(c)$, with $c'_{i}\in\{c_{i},\frac{c_{i}}{4}\}$, whenever $(\R \s H^n(c'_{i}))\leq M_{i}$ for every $i\in\{1,\ldots,r\}$.
	\item $\mathbb{F}\s H^n(c)$, with $\mathbb{F}\neq \R$  and $c'_{i}=c_{i}$,  whenever $(\mathbb{F} \s H^n(c_{i}))\leq M_i$, for every $i\in\{1,\ldots,r\}$.
\end{enumerate}
\end{corollary}
\begin{proof}
Let $\Sigma\subset M$ be a non-flat, $r$-diagonal, irreducible, totally geodesic submanifold. We will proceed by induction on $r$. The statement is true for $r=1$ by the classification of totally geodesic submanifolds in symmetric spaces of non-compact type and rank one (see Table \ref{table:rankone}). Let us assume that it is true for $r-1$ factors and we will prove it for $r$.  Observe that $\Sigma$ is contained in $\widehat{\Sigma}\times\widehat{\Sigma}'$  where $\widehat{\Sigma}$ is the projection of $\Sigma$ over $M_{1}\times\cdots\times M_{r-1}$ and $\widehat{\Sigma}'$ is the projection of $\Sigma$ over $M_{r}$. By Lemma \ref{lemma:projtriple}, we have that $\widehat{\Sigma}$ and $\widehat{\Sigma}'$ are totally geodesic submanifolds of $M$. Furthermore, by Lemma \ref{lemma:kblock}~\textit{iii)} and Remark~\ref{remark:k-diagonaltotgeod}, we have that $\widehat{\Sigma}$ and $\widehat{\Sigma}'$ are homothetic to $\Sigma$. Moreover, $\widehat{\Sigma}$ is $(r-1)$-diagonal in $M$, since $\Sigma$ is $r$-diagonal in $M$, and by induction hypothesis, $\widehat{\Sigma}$ is isometric to $\mathbb{F} \s H^n(\widetilde{c})$ for $\mathbb{F}\in\{\R,\C,\H,\O\}$, where $$\widetilde{c}=\frac{\prod_{i=1}^{r-1} c'_{i} }{e_{r-2}(c'_{1},\ldots,c'_{r-1})},$$ with $c'_{i}\in\{c_{i},c_{i}/4\}$ for each $i\in\{1,\ldots, r-1\}$. Let us assume that $\mathbb{F}=\R$,
 since the result follows similarly in the other cases. Then, by Theorem \ref{th:rank2reducible}, since $\Sigma$ is $2$-diagonal in $\widehat{\Sigma}\times \widehat{\Sigma}'$, and $\widehat{\Sigma}$ and  $\widehat{\Sigma}'$ are of rank one, the sectional curvature of $\Sigma$ is equal to the opposite of
\[ c=\frac{\widetilde{c}\hspace{0.1cm} c'_{r} }{\widetilde{c}+ c'_{r}}=\frac{\prod_{i=1}^r c'_{i}}{e_{r-1}(c'_{1},\ldots,c'_{r})},  \]
where we have used $e_{r-1}(X_1,\ldots,X_r)=e_{r-1}(X_1,\ldots,X_{r-1})+e_{r-2}(X_1,\ldots,X_{r-1})X_{r}$, for arbitrary variables $X_1,\ldots,X_r$.\qedhere
\end{proof}
\begin{figure}[h]
	\label{fig:youngtableaux}
	\begin{minipage}{1\textwidth}\small
		\hspace{0.23cm}
		\begin{tabular}{|c|c|c|l} 
			\cline{1-3}
			$\mathbb{R}\s H^3(c_1)\subset\mathbb{R}\s H^3(c_1)$ & $\mathbb{R}\s H^3(c_2/4)\subset\mathbb{C}\s H^3(c_2) \hspace{0.1cm}$ & $\mathbb{R}\s H^3(c_3)\subset\mathbb{H}\s H^3(c_3)\hspace{0.1cm} $ & $\R \s H^3\left(\frac{c_1c_2c_3}{c_1c_2+4c_1c_3+c_2c_3}\right)$ \\ \cline{1-3}
		\end{tabular}
	\end{minipage}%
	
	\vspace{0.3cm}
	\hspace{-7cm}
	\renewcommand{\arraystretch}{1.2}
	\begin{minipage}{1\textwidth}\small
		\hspace{3.6cm}
		\begin{tabular}{|l|cl}
			\cline{1-2}
			$\C \s H^2(c_2)\subset\mathbb{C}\s H^3(c_2) \hspace{0.1cm} $    & \multicolumn{1}{c|}{$\C \s H^2(c_3)\subset\mathbb{H}\s H^3(c_3) \hspace{0.1cm}$} & $\C \s H^2\left(\frac{c_2 c_3}{c_2+c_3}\right)$ \\ \cline{1-2}
			\multicolumn{1}{|c|}{$\R \s H^2(c_1)\subset\mathbb{R}\s H^3(c_1)$} &                                                              & \hspace{-3.4cm}$\R \s H^2(c_1)$                           \\ \cline{1-1}
		\end{tabular}%
	\end{minipage}
	
	\vspace{0.3cm}
	\hspace{-13.55cm}
	
	\begin{minipage}{1\textwidth}\small
		\hspace{0.25cm}
		\begin{tabular}{|c|c}
			\cline{1-1}
			$\mathbb{R}\s H^3(c_1)\subset \mathbb{R}\s H^3(c_1)$                                          & \multicolumn{1}{l}{$\mathbb{R}\s H^3(c_1)$ }                  \\ \cline{1-1}
			\multicolumn{1}{|l|}{$\mathbb{R}\s H^3(c_2/4)\subset\mathbb{C}\s H^3(c_2) \hspace{0.1cm} $} &   $\mathbb{R}\s H^3(c_2/4)$  \\ \cline{1-1}
			$\R \s H^4(c_3)\subset\mathbb{H}\s H^3(c_3) \hspace{0.1cm} $                       &\hspace{-0.4cm} $\R \s H^4(c_3)$ \\ \cline{1-1}
		\end{tabular}%
	\end{minipage}

	\caption{Three examples of Young tableaux adapted to the product  $M=\R \s H^3(c_1)\times\C \s H^3(c_2)\times \H \s H^3(c_3)$ and the corresponding totally geodesic submanifolds represented by each row (see Proposition~\ref{prop:youngtableaux}). Notice that the isometry type of these totally geodesic submanifolds can be computed using Corollary~\ref{cor:sympolynomial}.}
\end{figure}


Now we will provide the classification of totally geodesic submanifolds in products of symmetric spaces of rank one by introducing a combinatorial object that we will call adapted Young tableau.

We first recall the well-known notions of partition of a positive integer and of Young diagram. Let $r\ge1$ be a positive integer. Then, a \textit{partition} of $r$ is a vector $\lambda=(\lambda_1,\ldots,\lambda_k)\in \mathbb{Z}^k$ such that $r=\sum_{j=1}^k \lambda_j$ and $\lambda_1\ge\ldots\ge \lambda_k\ge 1$. For each partition $\lambda$ we associate a \textit{Young diagram}. This is a collection of boxes with $\lambda_j$ boxes in the $j$-th row, for each $j\in\{1,\ldots,k\}$.

Now we will introduce the notion of Young tableau adapted to a product $M$ of symmetric spaces of non-compact type and rank one. Let us consider $M=\mathbb{F}_1 \s H^{n_1}(c_1) \times \cdots \times \mathbb{F}_r \s H^{n_r}(c_r)$,
where $n_i\ge2$, $c_i>0$ and $\mathbb{F}_i\in\{\R,\C,\H,\O\}$ for each $i\in\{1,\ldots,r\}$.  Let $\lambda=(\lambda_1,\ldots,\lambda_k)$ be a partition of $r$ and let us consider its Young diagram. We will add to the $m$-th box in the $j$-th row a totally geodesic inclusion $\mathbb{F}'_{i_{j,m}}\s H^{n'_{i_{j,m}}}(c'_{i_{j,m}})\subset\mathbb{F}_{i_{j,m}}\s H^{n_{i_{j,m}}}(c_{i_{j,m}})$, for every $j\in\{1,\ldots,k\}$ and $m\in\{1,\ldots,\lambda_j\}$, where $\bigsqcup_{j=1}^k\{i_{j,1},\ldots,i_{j,\lambda_j}\}=\{1,\ldots,r\}$.
Furthermore, we require all totally geodesic submanifolds appearing in the $j$-th row to be mutually homothetic.
A Young diagram with this information will be called \textit{Young tableau adapted to $M$}. See Figure 1 for various examples of this.

The following proposition is crucial for the proof of Theorem~A.
\begin{proposition}
\label{prop:youngtableaux}
Let $M=M_1\times \cdots\times M_r$, where $M_i$ is a symmetric space of non-compact type and rank one for each $i\in\{1,\ldots,r\}$. Then, the following statements hold:
\begin{enumerate}
	\item[i)] For each Young tableau $T$ adapted to $M$ we can attach a set $\mathcal{S}(T)$ of semisimple totally geodesic submanifolds $\Sigma_T$ of $M$ which have non-trivial projection onto each factor of $M$.
	\item[ii)] If $\Sigma_T$ and $\widetilde{\Sigma}_T$ belong to $\mathcal{S}(T)$, then $\Sigma_T$ is isometric to $\widetilde{\Sigma}_T$.
	\item[iii)] If $\Sigma$ is a semisimple totally geodesic submanifold of $M$ which has non-trivial projection onto each factor of $M$, then it is equal to some $\Sigma_T\in\mathcal{S}(T)$ for some Young tableau $T$ adapted to $M$.
\end{enumerate}
\end{proposition}
\begin{proof}
First of all, we will see how to construct a totally geodesic submanifold of $M$ from a Young tableau adapted to $M$. Let $T$ be a Young tableau adapted to $M$ and let us assume that it has $k$ rows. Let us further assume that it has $\lambda_j$ boxes in the $j$-th row. Namely, $\mathbb{F'}_{i_{j,1}}\s H^{n'_{i_{j,1}}}(c'_{i_{j,1}})\subset \mathbb{F}_{i_{j,1}}\s H^{n_{i_{j,1}}}(c_{i_{j,1}}), \ldots,\mathbb{F'}_{i_{j,\lambda_j}}\s H^{n'_{i_{j,\lambda_j}}}(c'_{i_{j,\lambda_j}})\subset\mathbb{F}_{i_{j,\lambda_j}}\s H^{n_{i_{j,\lambda_j}}}(c_{i_{j,\lambda_j}})$ are the labels in the boxes in the $j$-th row, where $\bigsqcup_{j=1}^k\{i_{j,1},\ldots,i_{j,\lambda_j}\}=\{1,\ldots,r\}$. Let $\g{p}_{i_{j,m}}$ be a Lie triple system corresponding to $M_{i_{j,m}}=\mathbb{F}_{i_{j,m}}\s H^{n_{i_{j,m}}}(c_{i_{j,m}})$, for each $m\in\{1,\ldots,\lambda_j\}$. Then, for each $m\in\{1,\ldots,\lambda_j\}$, there is some Lie triple system ${\g{p}}_{i_{j,m}}'\subset \g{p}_{i_{j,m}}$ which corresponds to the totally geodesic embedding in the $m$-th box of the $j$-th row of $T$. Notice that, by construction of $T$, these totally geodesic submanifolds are mutually homothetic. Let us define ${\g{g}}_{i_{j,m}}':={\g{p}}_{i_{j,m}}'\oplus[{\g{p}}_{i_{j,m}}',{\g{p}}_{i_{j,m}}']$.  Clearly, for any fixed $j\in\{1,\ldots,k\}$, all these Lie algebras ${\g{g}}_{i_{j,m}}'$, $m\in\{1,\ldots,\lambda_j\}$, are mutually isomorphic because they are the Lie algebras of the isometry groups of mutually homothetic semisimple symmetric spaces. Let $\varphi^j_{1,m}\colon {\g{g}}_{i_{j,1}}'\rightarrow{\g{g}}_{i_{j,m}}'$ be a Lie algebra isomorphism sending ${\g{p}}_{i_{j,1}}'$ into ${\g{p}}_{i_{j,m}}'$ for $m\in\{2,\ldots,\lambda_j\}$, and $\varphi^j_{1,1}$ be the identity map of $\g{g}_{i_{j,1}}'$. Now, we define $ \widehat{\g{p}}_j:=\left\{\sum_{m=1}^{\lambda_j} \varphi^j_{1,m} X: X\in \g{p}'_{i_{j,1}}\right\}$
for each $j\in\{1,\ldots,k\}$.
Then, $\widehat{\g{p}}_j$ is a Lie triple system in $\g{p}$ which is $\lambda_j$-diagonal.

We perform this process for each row $j\in\{1,\ldots,k\}$ to define
$ {\g{p}}_{T}:= \bigoplus_{j=1}^k \widehat{\g{p}}_j$.
By construction, we have that $[\widehat{\g{p}}_j ,\widehat{\g{p}}_{j'} ]=0$ for $j\neq j'$ in $\{1,\ldots,k\}$. Hence, ${\g{p}}_{T}$ is a Lie triple system in $\g{p}$ and we will denote by $\Sigma_T=\Sigma_1\times\cdots\times\Sigma_k$ its corresponding totally geodesic submanifold, where $\Sigma_j$ is the totally geodesic submanifold corresponding to $\widehat{\g{p}}_j$. Consequently, for a Young tableau $T$  adapted to $M$, we have constructed a semisimple totally geodesic submanifold $\Sigma_T$ of $M$, that has non-trivial projection onto each factor of $M$. However, notice that this construction depends on the Lie triple system $\g{p}'_{i_{j,m}}$ in the subspace $\g{p}_{i_{j,m}}$ and on the Lie algebra isomorphism $\varphi^j_{1,m}$ that we chose, and if we choose different Lie triple systems and isomorphisms, we get different semisimple totally geodesic submanifolds with non-trivial projection onto each factor of $M$. Hence, for each Young tableau $T$ adapted to $M$ we can attach a set $\mathcal{S}(T)$ which is equal to the set of all the totally geodesic submanifolds which can be constructed from $T$ through the process described above. This proves \textit{i)}.

Now we will check that all totally geodesic submanifolds in $\mathcal{S}(T)$ are  mutually isometric. First of all, observe that two totally geodesic submanifolds of a symmetric space of rank one are congruent if and only if they are isometric. Hence, the totally geodesic embedding $\mathbb{F}'_{i_{j,m}} \s H^{n'_{i_{j,m}}}(c'_{i_{j,m}})\subset \mathbb{F}_{i_{j,m}} \s H^{n_{i_{j,m}}}(c_{i_{j,m}})$ represents a congruence class of totally geodesic embeddings in $M_{i_{j,m}}=\mathbb{F}_{i_{j,m}} \s H^{n_{i_{j,m}}}(c_{i_{j,m}})$. Let $N_{i_{j,m}}$ and $\widetilde{N}_{i_{j,m}}$ be congruent totally geodesic submanifolds in $M_{i_{j,m}}$ corresponding to the inclusion in the $m$-th box of the $j$-th row for each $m\in\{1,\ldots,\lambda_j\}$ and $j\in\{1,\ldots,r\}$. Then, there is a isometry $\varphi_{i_{j,m}}$ of $M_{i_{j,m}}$ such that $\varphi_{i_{j,m}}N_{i_{j,m}}=\widetilde{N}_{i_{j,m}}$. Following the above procedure, we obtain two different $\lambda_j$-diagonal totally geodesic submanifolds, namely,
\[\Sigma_j\subset N_{i_{j,1}}\times\cdots \times N_{i_{j,\lambda_j}}\quad \text{and} \quad \widetilde{\Sigma}_j\subset \widetilde{N}_{i_{j,1}}\times\cdots \times \widetilde{N}_{i_{j,\lambda_j}},    \]
	such that $\Sigma_j$ and $\widetilde{\Sigma}_j$ are homothetic to $N_{i_{j,m}}$ and $\widetilde{N}_{i_{j,m}}$ for every $m\in\{1,\ldots,\lambda_j\}$ and for each $j\in\{1,\ldots,k\}$. Now let $\varphi_j:=(\varphi_{i_{j,1}},\ldots,\varphi_{i_{j,\lambda_j}})\colon M_{i_{j,1}}\times\cdots\times M_{i_{j,\lambda_j}}\rightarrow M_{i_{j,1}}\times\cdots\times M_{i_{j,\lambda_j}}$. Then, $\varphi_j$ is an isometry of $M_{i_{j,1}}\times\cdots\times M_{i_{j,\lambda_j}}$ for every $j\in\{1,\ldots,k\}$. Moreover, $\varphi_j \Sigma_j$ is a non-flat, irreducible, $\lambda_j$-diagonal totally geodesic submanifold in $\widetilde{N}_{i_{j,1}}\times\cdots \times \widetilde{N}_{i_{j,\lambda_j}}$. By Proposition~\ref{prop:congruency}, there is $\psi_j\in \Isom(\widetilde{N}_{i_{j,1}})\times\cdots\times\Isom(\widetilde{N}_{i_{j,\lambda_j}})$ such that  $\psi_j\varphi_j \Sigma_j=\widetilde{\Sigma}_j$.  Consequently, $\Sigma=\Sigma_1\times \cdots\times\Sigma_k$ is isometric to $\widetilde{\Sigma}=\widetilde{\Sigma}_1\times\cdots\times\widetilde{\Sigma}_k$. This proves \textit{ii)}.

Let $\Sigma\subset M$ be a semisimple totally geodesic submanifold that has non-trivial projection onto each factor of $M$. Let $k\leq r$ be the rank of $\Sigma$. By combining De-Rham Theorem, Lemma \ref{lemma:kblock} and Lemma \ref{lemma:proj}, we can ensure the existence of a partition $\{ \{i_{j,1},\ldots,i_{j,\lambda_j}\}: j\in\{1,\ldots,k\}\}$ of $\{1,\ldots,r\}$ satisfying:
\begin{itemize}
	\item $\Sigma=\Sigma_1\times \cdots\times \Sigma_k$, where $\Sigma_j$ has rank one for every $j\in\{1,\ldots,k\}$.
	\item $\Sigma_j\subset M$ is $\lambda_j$-diagonal for every $j\in\{1,\ldots,k\}$.
	\item $r=\sum_{j=1}^k \lambda_j$.
	\item $\Sigma_j\subset N_j:=N_{i_{j,1}}\times\cdots\times N_{i_{j,\lambda_j}}$, where $N_{i_{j,m}}\subset M_{i_{j,m}}$ is a totally geodesic submanifold homothetic to $\Sigma_j$ for each $j\in\{1,\ldots,k\}$ and $m\in\{1,\ldots,\lambda_j\}$.
\end{itemize}
 Then, up to some reordering, we can assume that $(\lambda_1,\ldots,\lambda_k)$ is a partition of $r$. Let $T$ be the Young diagram associated with $(\lambda_1,\ldots,\lambda_k)$. Then, $\Sigma_j$ projects non-trivially exactly onto $\lambda_j$ factors of $M$. We fill the $\lambda_j$ boxes of the $j$-th row of $T$ with the  data corresponding to the totally geodesic inclusion resulting of projecting $\Sigma_j$ in each of these factors. Therefore, we can find a Young tableau $T$ adapted to $M$ such that $\Sigma=\Sigma_1\times\cdots\times \Sigma_k$ is equal to some $\Sigma_{T}$ in $\mathcal{S}(T)$. This proves \textit{iii)}.\qedhere \end{proof}
 We are now ready to prove the first main theorem of this article.
\begin{proof}[Proof of Theorem~A]
By Proposition~\ref{prop:youngtableaux}, it follows that $\Sigma_0\times\Sigma_T$ is a totally geodesic submanifold of $M$, where $\Sigma_T$ is a semisimple totally geodesic submanifold corresponding to a Young tableau $T$ adapted to $M_{\sigma(1)}\times \cdots\times M_{\sigma(k)}$, $\Sigma_0$ is a flat totally geodesic submanifold of $M_{\sigma({k+1})}\times \cdots \times M_{\sigma(r)}$,  $\sigma$ is any permutation of $\{1,\ldots,r\}$, and $k\in\{1,\ldots,r\}$.

Let $\Sigma$ be a totally geodesic submanifold of $M$. By De-Rham Theorem we have $\Sigma=\Sigma_0\times \Sigma_1$, where $\Sigma_0$ is flat and $\Sigma_1$ is semisimple. Then, $\Sigma_1$ projects non-trivially onto $M_{\sigma(1)}\times\cdots\times M_{\sigma(k)}$ for some  $k\in\{1,\ldots,r\}$ and some permutation $\sigma$ of $\{1,\ldots,r\}$. Thus, by Lemma \ref{lemma:proj}, we have $\Sigma_0\subset M_{\sigma(k+1)}\times \cdots\times M_{\sigma(r)}$. Now, by Proposition \ref{prop:youngtableaux}, there is a Young tableau $T$ adapted to $M_{\sigma(1)}\times\cdots\times M_{\sigma(k)}$ such that $\Sigma_1$ is equal to some $\Sigma_T$ in $\mathcal{S}(T)$. Therefore, $\Sigma$ is equal to $\Sigma_0\times\Sigma_T$ as submanifolds in $M$.
\end{proof}

\section{Totally geodesic submanifolds in Hermitian symmetric spaces}
\label{sect:slant}

In this section, we will construct infinitely many examples of totally geodesic submanifolds in Hermitian symmetric spaces which have constant K\"ahler angle different from $0$ or~$\pi/2$.

We start  by recalling the notion of K\"ahler angle (see \cite{damekricci} or \cite{mathz}). Let us equip the complex vector space  $\mathbb{C}^n$ with $\langle \cdot,\cdot\rangle$, the scalar product given by considering the real part of its standard Hermitian scalar product, and denote the complex structure of $\mathbb{C}^n$ by $J$ (multiplication by the imaginary unit). Furthermore, let us  consider $V\subset \mathbb{C}^n$ a real vector subspace and  the orthogonal projection $\pi_V\colon \mathbb{C}^n\rightarrow V$  onto $V$. The K\"ahler angle of a non-zero $v\in V$ with respect to $V$ is defined as the value $\varphi\in[0,\pi/2]$ such that $\langle \pi_V J v, \pi_V J v\rangle=\cos^2(\varphi)\langle v,v\rangle$. We say that a real subspace $V\subset \mathbb{C}^n$ has \textit{constant K\"ahler angle} $\varphi\in[0,\pi/2]$ if the K\"ahler angle of every non-zero vector $v\in V$ is $\varphi$. In particular,  $V\subset \mathbb{C}^n$ has constant K\"ahler angle equal to $0$ if and only if it is a complex subspace, and it has constant K\"ahler angle equal to $\pi/2$ if and only if it is totally real. Also, a submanifold $\Sigma$ in a K\"ahler manifold $M$ is said to have constant K\"ahler angle $\varphi\in[0,\pi/2]$ if the tangent space of $\Sigma$ at each point is a subspace with constant K\"ahler angle $\varphi$ in the corresponding tangent space of $M$. In the setting of Hermitian symmetric spaces, since totally geodesic submanifolds are homogeneous and the isometries that belong to the connected component of the identity of the isometry group are holomorphic, the previous property
needs to be checked only at one point.

Now we will recall some known facts about complex Grassmannians which will be useful in this section. Let $M=\s G_k(\C^{n+k})$ be the Grassmannian of complex $k$-planes in $\C^{n+k}$. Then, we have  $M=\s G/\s K$, with $\s G=\mathsf{SU}_{n+k}$ and $\s K=\mathsf{S}(\mathsf{U}_k\times\mathsf{U}_n)$. We can decompose $\mathfrak{su}_{n+k}$ as  $\mathfrak{su}_{n+k}=\g{k}\oplus\g{p}$ where
\begin{equation}
\label{eq:cartandec}
\begin{aligned}
 \g{k}&=\left\{ \left(
 \begin{array}{c|c}
 A & 0 \\
 \hline
 0 & B
 \end{array}
 \right): A\in \g{u}_k, B\in\g{u}_n, \tr A + \tr B=0 \right\}, \\
 \g{p}&=\left\{
 \left(
 \begin{array}{c|c}
 0 & X \\
 \hline
 -X^* & 0
 \end{array}
 \right): X \in \mathcal{M}_{k,n}(\mathbb{C})  \right\}.
 \end{aligned}
\end{equation}
Furthermore, $M=\s G/\s K$ is a Hermitian symmetric space. Then,   $\g{p}$ inherits a complex structure $J$ which is given by $\ad_Z\vert_\g{p}$ for some $Z\in Z(\g{k})$ (see \cite[Theorem 7.117]{knapp}).
Hence, we can identify $\g{p}$ with the complex vector space $\C^k\otimes \C^n$ by the usual isomorphism with matrices with complex entries. 

For $k=1$, we have $\s G_1(\C^{n+1})$, which is the complex projective space $\C \s P^n$. In this case the Lie algebra of the isometry group is $\g{su}_{n+1}$.  Let $\Theta\colon \g{su}_{n+1}\rightarrow \g{su}_{n+1}$ be such that $X\in \g{su}_{n+1}$ is mapped to $\overline{X}\in \g{su}_{n+1}$, the complex conjugate of $X$. Clearly, $\Theta$ is a Lie algebra automorphism of $\g{su}_{n+1}$ and  preserves $\g{p}$. Furthermore, let $\{e_1,\ldots, e_n\}$ be the canonical basis for $\g{p}\equiv \C^n$. Observe that we have
\begin{equation}
\label{eq:conjugation}
 \Theta e_i = e_i, \qquad \Theta J e_i=- J e_i, 
\end{equation}
for every $i=1,\ldots,n$.

\begin{remark}
	Notice that $\C\s P^1$ is identified with the round sphere of dimension $2$ and constant sectional curvature equal to one.
\end{remark}

\begin{theorem}
\label{th:slant}
Let $M=\C\s P^{n}\times\stackrel{k)}{\cdots}\times\C \s P^n$. Then, for each $s\in\{0,\ldots,k\}$, there is a $k$-diagonal totally geodesic submanifold homothetic to $\C\s P^n$ in $M$ with constant K\"ahler angle $\varphi\in [0,\pi/2]$ satisfying 
\[  \cos(\varphi)=\left|\frac{2s-k}{k}\right|.\]  
\end{theorem}
\begin{proof}
First of all, let $\g{p}_i$ be the Lie triple system corresponding to the $i$-th factor of $M$.  Then, $\g{g}_i=\g{p}_i\oplus \g{k}_i\simeq \g{su}_{n+1}$, where $\g{k}_i:=[\g{p}_i,\g{p}_i]$. Moreover, consider the complex structure $J=(J_1,\ldots,J_k)$ of $\bigoplus_{i=1}^k \g{p}_i$, where $J_i$ is the complex structure induced by some element in $Z(\g{k}_i)$. Let $\sigma_{i}\colon \g{g}_1\rightarrow \g{g}_i$ be a Lie algebra isomorphism for $i\in\{2,\ldots,k\}$, and $\sigma_1:=\Id_{\mathfrak{g}_1}$. Furthermore, let us assume that $\sigma_i$ restricted to $\g{k}_1$ induces an isomorphism between $\g{k}_1$ and $\g{k}_i$. Thus, we have $\sigma_i\ad_{Z} X=\ad_{\sigma_i Z} \sigma_i X$,
for each $X\in \g{p}_1$ and $Z\in Z(\g{k}_1)$. In particular, by the discussion above and taking into account that $\dim Z(\g{k}_i)=1$, we have $\sigma_i J_1 X~=\pm J_i\sigma_i X$, for every $X\in \g{p}_1$, where we denote by $J_1$ and $J_i$ the complex structures of both $\g{p}_1$ and $\g{p}_i$, respectively. We will declare $\sigma_i=\sigma^+_i$, when $\sigma_i J_1(\sigma_i)^{-1}=J_i$, and $\sigma_i=\sigma^-_i$, when $\sigma_i J_1(\sigma_i)^{-1}=-J_i$. Additionally, notice that $\sigma^-_i=\sigma^+_{i}\circ \Theta$. Let us consider
\[  \g{p}_{\Sigma^s}:=\left\{ \sum_{i=1}^s \sigma^+_i   X + \sum_{i=s+1}^k\sigma^{-}_{i} X: X\in \g{p}_1   \right\},  \]
for each $s\in\{0,\ldots,k\}$, and let $\Sigma^s$ be its corresponding totally geodesic submanifold in $M$.
Clearly, this is a Lie triple system and it must correspond to some totally geodesic submanifold $\Sigma^s$ which is homothetic to $\C\s P^n$ by Corollary \ref{cor:sympolynomial} via duality (see also Remark~\ref{remark:k-diagonaltotgeod}). Let us choose  a $\C$-orthonormal basis $\{e_i\}_{i=1}^n$ for $\g{p}_1$ satisfying Equation (\ref{eq:conjugation}). Then, taking into account that  $\sigma^+_{i}$ restricted to $\g{k}_1$ gives an isomorphism onto $\g{k}_i$, and Equation (\ref{eq:conjugation}),
we can express $\g{p}_{\Sigma^s}$ as $\spann_{\R}\{v_i, w_i\}_{i=1}^n$, where
\[v_i:= \frac{1}{\sqrt{k}} \sum_{j=1}^k \sigma^+_{j} e_i, \qquad w_i:= \frac{1}{\sqrt{k}} \sum_{j=1}^s \sigma^+_{j}  J_1 e_i - \frac{1}{\sqrt{k}} \sum_{j=s+1}^k \sigma^+_j J_1 e_i  \]
constitute an orthonormal basis of $\g{p}_{\Sigma^s}$.
Let $v\in \g{p}_{\Sigma^s}$ be a unit vector. Then, we can write  $v=\sum_{i=1}^n \left( a_i v_i + b_i w_i\right)$, where $\{a_i\}_{i=1}^n$ and $\{b_i\}_{i=1}^n$ are real numbers satisfying that $\sum_{i=1}^n \left( a_i^2+b_i^2\right)=1$.
Now we have 
\begin{equation*}
\pi_{\g{p}_{\Sigma^s}} J v=\pi_{\g{p}_{\Sigma^s}}\bigg(\sum_{j=1}^n (a_i J v_i + b_i J w_i) \bigg)=\sum_{i=1}^{n} (a_i \pi_{\g{p}_{\Sigma^s}} J v_i + b_i \pi_{\g{p}_{\Sigma^s}} J w_i),
\end{equation*}
where $\pi_{\g{p}_{\Sigma^s}}$ denotes the orthogonal projection onto $\g{p}_{\Sigma^s}$. Moreover,
\begin{align*}
\pi_{\g{p}_{\Sigma^s}} J v_i&=\sum_{j=1}^n (\langle Jv_i, v_j\rangle v_j + \langle J v_i, w_j\rangle w_j)= \langle J v_i, w_i\rangle w_i\\
\pi_{\g{p}_{\Sigma^s}} J w_i&=\sum_{j=1}^n( \langle Jw_i, v_j\rangle v_j + \langle J w_i, w_j\rangle w_j)= \langle J w_i, v_i\rangle v_i,
\end{align*}
for each $i\in\{1,\ldots,n\}$, where we have used that $\sigma_j^+$ is a linear isometry, since all the factors in $M$ are mutually isometric. Furthermore, using again that $\sigma^+_j$ is a linear isometry that commutes with~$J_j$, we get
\[\langle J v_i, w_i\rangle= \frac{1}{k} \left( \sum_{j=1}^s \langle J_j \sigma^+_j  e_i,  \sigma^+_j J_1 e_i\rangle - \sum_{j=s+1}^{k} \langle J_j \sigma^+_{j}  e_i,\sigma^+_{j} J_1 e_i \rangle \right)=\frac{2s-k}{k}.  \]
Consequently, combining these equations, we conclude 
\begin{align*}
\langle \pi_{\g{p}_{\Sigma^s}} J v, \pi_{\g{p}_{\Sigma^s}} J v\rangle= \sum_{i=1}^n( a_i^2 \langle J v_i, w_i\rangle^2 + b_i^2 \langle J w_i, v_i\rangle^2)= \sum_{i=1}^n (a_i^2+b_i^2)\langle J v_i,w_i\rangle^2=\left(\frac{2s-k}{k}\right)^2.
\end{align*}
Since $v$ is an arbitrary unit vector in $\g{p}_{\Sigma^s}$, we have that $\Sigma^s$ has constant K\"ahler angle equal to $\varphi\in [0,\pi/2]$, where $\cos(\varphi)=\left|\frac{2s-k}{k}\right|$.
\end{proof}


Thus, Theorem \ref{th:slant} gives a method to construct totally geodesic submanifolds with constant K\"ahler angle in Hermitian symmetric spaces. We only need to find a product of complex projective spaces embedded in a totally geodesic way and use Theorem \ref{th:slant}. In particular, in the complex Grassmannians these products are very abundant as the following lemma shows.

\begin{lemma}
\label{lemma:grassman}
Let $(n_1,\ldots,n_k)$ be a partition of $n$. Then, there is a complex totally geodesic submanifold homothetic to $\C\s P^{n_1}\times \cdots\times \C\s P^{n_k}$ in $\s G_k(\C^{n+k})$.
\end{lemma}
\begin{proof}
Let $\Sigma=\C\s  P^{n_1}\times \cdots\times \C\s P^{n_k}$ and $(n_1,\dots,n_k)$ a partition of $n$. This means that $n_1\ge \ldots \ge n_k$ and $\sum_{i=1}^k n_i=n$. Now, for each $i\in\{1,\ldots, k\}$, we define the subspace 
\[\g{p}_i:=\spann_{\mathbb{C}}\{e_i\otimes e_{1+\sum_{j=1}^{i-1} n_{j}}, \ldots, e_i \otimes e_{\sum_{j=1}^i n_j} \},   \]
 where $\{e_l\otimes e_m\}_{l,m=1}^{k,n}$ is the canonical basis of $\C^k\otimes\C^n$.

Let us define $\g{p}_\Sigma:=\bigoplus_{i=1}^k \g{p}_i$. We will see that $\g{p}_{\Sigma}$ is a Lie triple system corresponding to $\Sigma$. Using the Lie bracket of $\g{su}_{n+k}$ and the description of $\g{p}$ in Equation (\ref{eq:cartandec}), it can be checked that
$\g{k}_i:=[\g{p}_i,\g{p}_i]\simeq \g{s}(\g{u}_1\times\g{u}_{n_i})$ and that $\g{p}_i$ is a $\g{k}_i$-module. Furthermore, it can be seen that $\g{g}_i=\g{k}_i\oplus \g{p}_i$ is isomorphic to $\g{su}_{n_i+1}$. This implies that $\g{p}_i$ is a Lie triple system corresponding to $\C \s P^{n_i}$. Moreover, it is clear that $[\g{p}_i,\g{p}_j]=0$ for $i\neq j$. Observe that $\g{p}_i$ is invariant under $\ad_Z$, where $Z\in Z(\g{k})$, and hence every $\g{p}_i$ is invariant under the complex structure of $\s G_k(\C^{n+k})$. Furthermore, it is easy to check that there is some $k\in  \s K$ such that $\Ad(k)$ interchanges the rows of $\g{p}$, see Equation (\ref{eq:cartandec}).
Consequently, $\g{p}_{\Sigma}$ is also a Lie triple system whose associated totally geodesic submanifold is homothetic to $\C\s P^{n_1}\times\cdots\times\C \s P^{n_k}$ of $\s G_k(\C^{n+k})$ and invariant under the complex structure of $\s G_k(\C^{n+k})$. \qedhere
\end{proof}
We are now ready to prove the second main theorem of this article.

\begin{proof}[Proof of Theorem~B]
Using Lemma \ref{lemma:grassman} and Theorem \ref{th:slant},  given any $q\in [0,1]\cap \mathbb{Q}$ and $m\in\mathbb{N}$, there exist $k,n\in \mathbb{N}$ such that there is a totally geodesic submanifold homothetic to $\C\s P^m$ with constant K\"ahler angle $\arccos(q)\in[0,\pi/2]$ in $\s G_k(\C^{n+k})$. This implies our result.
\end{proof}

\end{document}